\documentclass[11pt,a4paper]{amsart}
\usepackage{amsmath,amsfonts,amsthm,amssymb}
\usepackage{a4wide}
\usepackage{color,graphicx}
\usepackage{enumerate}

\theoremstyle{plain}
\newtheorem{theorem}{Theorem}[section]
\newtheorem{lemma}[theorem]{Lemma}
\newtheorem{proposition}[theorem]{Proposition}
\newtheorem{corollary}[theorem]{Corollary}
\newtheorem*{question}{Question}
\newtheorem*{KDT}{Koebe distortion theorem}
\newtheorem*{KOT}{Koebe one-quarter theorem}
\newtheorem*{thmA}{Theorem A}
\newtheorem*{thmB}{Theorem B}
\newtheorem*{thmC}{Theorem C}
\newtheorem*{corD}{Corollary D}

\theoremstyle{definition}
\newtheorem{definition}[theorem]{Definition}

\theoremstyle{remark}

\newtheorem{remark}[theorem]{Remark}

\newcommand \C{\mathbb{C}}

\newcommand \N{\mathbb{N}}
\newcommand \R{\mathbb{R}}
\newcommand \D{\mathbb{D}}
\newcommand \Z{\mathbb{Z}}
\newcommand \Q{\mathbb{Q}}
\newcommand{\bd}{\partial}

\DeclareMathOperator{\Sing}{Sing}
\DeclareMathOperator{\diam}{diam}
\DeclareMathOperator{\dist}{dist}

\DeclareMathOperator{\asinh}{asinh}

\newcommand{\dimhyp}{\dim_{hyp}}
\newcommand{\clC}{\overline{\C}}
\renewcommand{\Re}{\textup{Re}}
\renewcommand{\Im}{\textup{Im}}

\begin{document}

\title[Dimension of boundaries of attracting basins of entire maps]{On the dimension of the boundaries of attracting basins of entire maps}

\author[K. Bara\'nski]{Krzysztof Bara\'nski}
\address{University of Warsaw, Institute of Mathematics, ul.~Banacha 2, 02-097 Warszawa, Poland} 
\email{baranski@mimuw.edu.pl}

\author[B. Karpi\'nska]{Bogus{\l}awa Karpi\'nska}
\address{Faculty of Mathematics and Information Science, Warsaw University of Technology, ul.~Koszy\-kowa 75, 00-662 Warszawa, Poland}
\email{boguslawa.karpinska@pw.edu.pl}

\author[D. Mart\'i-Pete]{David Mart\'i-Pete}
\address{Department of Mathematical Sciences, University of Liverpool,  Liverpool L69 7ZL, United Kingdom} 
\email{david.marti-pete@liverpool.ac.uk}

\author[L. Pardo-Sim\'on]{Leticia Pardo-Sim\'on}
\address{Departament de Matem\`atiques i Inform\`atica, Universitat de Barcelona, Gran Via de les Corts Catalanes 585, 08007 Barcelona, Spain
\newline  
and
\newline
Centre de Recerca Matem\`atica, Campus de Bellaterra, Edifici C, 08193 Bellaterra, Barcelona, Spain}
\email{lpardosimon@ub.edu}

\author[A. Zdunik]{Anna Zdunik}
\address{University of Warsaw, Institute of Mathematics, ul. Banacha 2, 02-097 Warszawa, Poland} 
\email{A.Zdunik@mimuw.edu.pl}

\thanks{This research was funded in whole or in part by the National Science Centre, Poland, grants no.~2023/49/B/ST1/03015 (AZ) and 2023/51/B/ST1/00946 (KB \& BK), and the Spanish State Research Agency, grant PID2023-147252NB}


\begin{abstract} Let $f\colon \C \to \C$ be a transcendental entire map from the Eremenko-Lyubich class~$\mathcal B$, and let $\zeta$ be an attracting periodic point of period $p$. We prove that the boundaries of components of the attracting basin of (the orbit of) $\zeta$ have hyperbolic (and, consequently, Hausdorff) dimension larger than~$1$, provided $f^p$ has an infinite degree on an immediate component $U$ of the basin, and the singular set of $f^p|_U$ is compactly contained in $U$. The same holds for the boundaries of components of the basin of a parabolic $p$-periodic point $\zeta$, under the additional assumption $\zeta \notin \overline{\Sing(f^p)}$. We also prove that if an immediate component of an attracting basin of an arbitrary transcendental entire map is bounded, then the boundaries of components of the basin have hyperbolic dimension larger than~$1$. This enables us to show that the boundary of a component of an attracting basin of a transcendental entire function is never a smooth or rectifiable curve. The results provide a partial answer to a question from Hayman's list of problems in function theory.
\end{abstract}

\maketitle

\section{Introduction and results}\label{sec:intro}

For an entire function $f\colon \C \to \C$, its set of stability $\mathcal F(f)$, known as the \emph{Fatou set}, is defined as the set of points $z \in \C$, for which the family $\{f^n\}_{n=1}^\infty$ of the forward iterates of~$f$ is normal in some neighbourhood of $z$; the \emph{Julia set} $J(f)$, which is the locus of chaotic behaviour of $f$, is the complement of $\mathcal F(f)$. The set $\mathcal F(f)$ is open, and each of its connected components is called a \emph{Fatou component} of~$f$. We say that $U$ is a \emph{periodic} Fatou component of period $p\ge 1$, if $f^p(U)\subset U$, and $U$ is \emph{invariant} if $f(U)\subset U$. If $U$ is a Fatou component which is mapped into a periodic component by a forward iterate of $f$, then we say that $U$ is \emph{preperiodic}; otherwise it is called a \emph{wandering domain}. In this paper we study the boundaries of periodic or preperiodic Fatou components, which are parts of \emph{attracting} or \emph{parabolic basins} of periodic points of $f$. Recall that a periodic point $\zeta \in \C$, with $f^p(\zeta) = \zeta$ for some $p \ge 1$, is \emph{attracting} if $|(f^p)'(\zeta)| < 1$, and \emph{parabolic} if $(f^p)'(\zeta) = e^{2\pi i \theta}$, $\theta \in \Q$. The \emph{attracting} (resp.~ \emph{parabolic}) \emph{basin} of the orbit of $\zeta$ (or, in short, of $\zeta$) is defined as the set of points $z \in \C$, such that $f^{pn}(z)\to f^j(\zeta)$ for some $j \in \{0, \ldots, p-1\}$ as $n\to\infty$. An \emph{immediate} component of an attracting (resp.~parabolic) basin is its (periodic) component $U$, such that $U$ (resp.~$\bd U$) contains an element of the orbit of $\zeta$, where $\bd U$ denotes the boundary of $U$.

The Hausdorff dimension (denoted by $\dim_H$) of the Julia sets of transcendental entire maps has been widely studied, starting from the seminal paper by McMullen \cite{mcmullen87} on the exponential family $E_\lambda(z) = \lambda e^z$, $\lambda \in \C \setminus \{0\}$. 
By a result of Baker \cite{baker75}, the boundary of every Fatou component $U$ of a transcendental entire function $f$ contains non-trivial continua, so $\dim_H \partial U\geq 1$, and therefore $\dim_H J(f)\geq 1$. It was a long-standing open question whether there exists a transcendental entire map $f$ with $\dim_H J(f)=1$. This was answered in the affirmative by Bishop \cite{bishop18}. In \cite{stallard91,stallard04}, Stallard presented examples of transcendental entire maps $f$ with $\dim_H J(f) = t$ for every given $t \in (1,2)$. As in \cite{mcmullen87} it was proved that $\dim_H J(E_\lambda) = 2$ for every $\lambda \in \C \setminus \{0\}$, one can see that the Hausdorff dimension of the Julia set of a transcendental entire function can attain any value in the interval $[1,2]$. 

The functions studied in \cite{stallard91,stallard04}, as well as the exponential maps, belong to the well-known \emph{Eremenko--Lyubich class} $\mathcal B$, that consists of all transcendental entire maps $f$, for which the set of its critical and asymptotic values, denoted by $\Sing(f)$, is bounded. In~\cite{stallard2}, Stallard showed that $\dim_H J(f)>1$ for every $f\in\mathcal{B}$. Note that $\dim_H J(f)=2$ for $f \in \mathcal B$ with a finite or `infinite, but not too large' order of growth, as proved in \cite{dimpar,bergweiler-karpinska-stallard09,schubert}.

In the context of the dynamics of holomorphic maps, it is natural to consider also another kind of dimension, namely the \emph{hyperbolic dimension}. Recall that the hyperbolic dimension of a set $A\subset J(f)$ is defined as
\[
\dimhyp A = \sup\{\dim_H X \colon X \subset A \textup{ is an expanding conformal repeller}\}.
\]
Here, by an expanding conformal repeller we mean a compact forward-invariant set $X$ such that $f$ is \emph{uniformly expanding} on $X$, that is, there exists $k\in\mathbb{N}$ with $|(f^k)'|> 1$ on $X$. Obviously, $\dimhyp A \leq \dim_H A$. However, the two dimensions need not coincide. Indeed, there exist transcendental entire maps $f$, for which $\dimhyp J(f)$ is an arbitrary small positive number (see \cite{bergweiler12}).
In \cite{urbanski-zdunik03}, Urba\'nski and Zdunik proved that for exponential maps $E_\lambda$ with an attracting periodic point, there holds $\dim_{hyp} J(E_\lambda) < \dim_H J(E_\lambda) = 2$. In fact, more examples of this kind can be found among \emph{hyperbolic} transcendental entire maps, that is, maps $f$ for which the closure of the \emph{post-singular set} $\overline{\bigcup_{n=0}^\infty f^n(\Sing(f))}$ is bounded and disjoint from $J(f)$ (see e.g.~\cite{stallard99}). On the other hand, there are examples of hyperbolic transcendental entire maps $f$, where $\dim_{hyp} J(f) = \dim_H J(f) = 2$ (see \cite{rempe-gillen14}). A useful fact is that the hyperbolic dimension of the Julia set of a hyperbolic transcendental entire map $f$ is equal to the (generalized) zero of the topological pressure function, as proved in \cite{bowen}. In \cite{logtract}, it was shown that the hyperbolic dimension of the Julia set is larger than~$1$ for transcendental entire maps from class $\mathcal B$ and, more generally, for maps with a logarithmic tract.

Much less is known about the Hausdorff and hyperbolic dimension of the boundaries of particular Fatou components, which are subsets of the Julia set. Note that if a Fatou component $U$ is completely invariant, then the Julia set of $f$ coincides with the boundary of~$U$ and, consequently, $\dim_H J(f) = \dim_H \bd U$. Such a situation occurs, for example, within the exponential family, if the map $E_\lambda$ has an attracting fixed point $\zeta$, which happens e.g.~for $\lambda \in (0,1/e)$. Recall that in this case $\dim_H J(f) = \dim_H \bd U = 2$, where $U$ is the unique component of the attracting basin of $\zeta$. Nevertheless, the dimension of the boundaries of Fatou components can be smaller than the dimension of the whole Julia set. In \cite{basexp}, Bara\'nski, Karpi\'nska and Zdunik proved that if the exponential map $E_\lambda$ has an attracting periodic point $\zeta$ of period $p>1$, then $1 < \dim_H \bd U < \dim_H J(E_\lambda) = 2$ for every component $U$ of the attracting basin of $\zeta$. However, outside the exponential family, little is known on the dimension of the boundaries of Fatou components of transcendental entire maps. 

In 1989, contributing to the famous Hayman's list of problems in function theory, Hamilton asked (\cite[Problem~2.76]{Hayman-problems1989}, see also the 2019 edition \cite{Hayman-problems2019}), whether the boundary of a Fatou component of an entire function has Hausdorff dimension larger than~$1$ unless it is a circle or a line. Given that examples of Siegel disks (that is, simply connected $p$-periodic Fatou components, where $f^p$ is conformally conjugate to an irrational rotation) and wandering domains of transcendental entire maps with boundaries of Hausdorff dimension equal to~$1$ have been found (see, respectively, \cite{avila-buff-cheritat04,geyer08} and  \cite{bishop18,bocthaler21}), the following question arises naturally.

\begin{question}
Let $U$ be a component of an attracting or parabolic basin of a transcendental entire function. Is it true that $\dim_H \partial U>1$? 
\end{question}

One can naturally distinguish three types of $p$-periodic (immediate) components $U$ of attracting or parabolic basins of entire transcendental maps $f$: the ones with $\deg f^p|_U = \infty$, which are necessarily unbounded, the unbounded ones with $\deg f^p|_U < \infty$, and the bounded ones, which always satisfy $\deg f^p|_U < \infty$ (by $\deg$ we denote the degree of a map, cf.~Proposition~\ref{prop:swiss}). All the three types of basins actually appear among maps from class~$\mathcal B$, see Figure~\ref{fig:three_types}.

\begin{figure}[ht!]
\includegraphics[width=.315\textwidth]{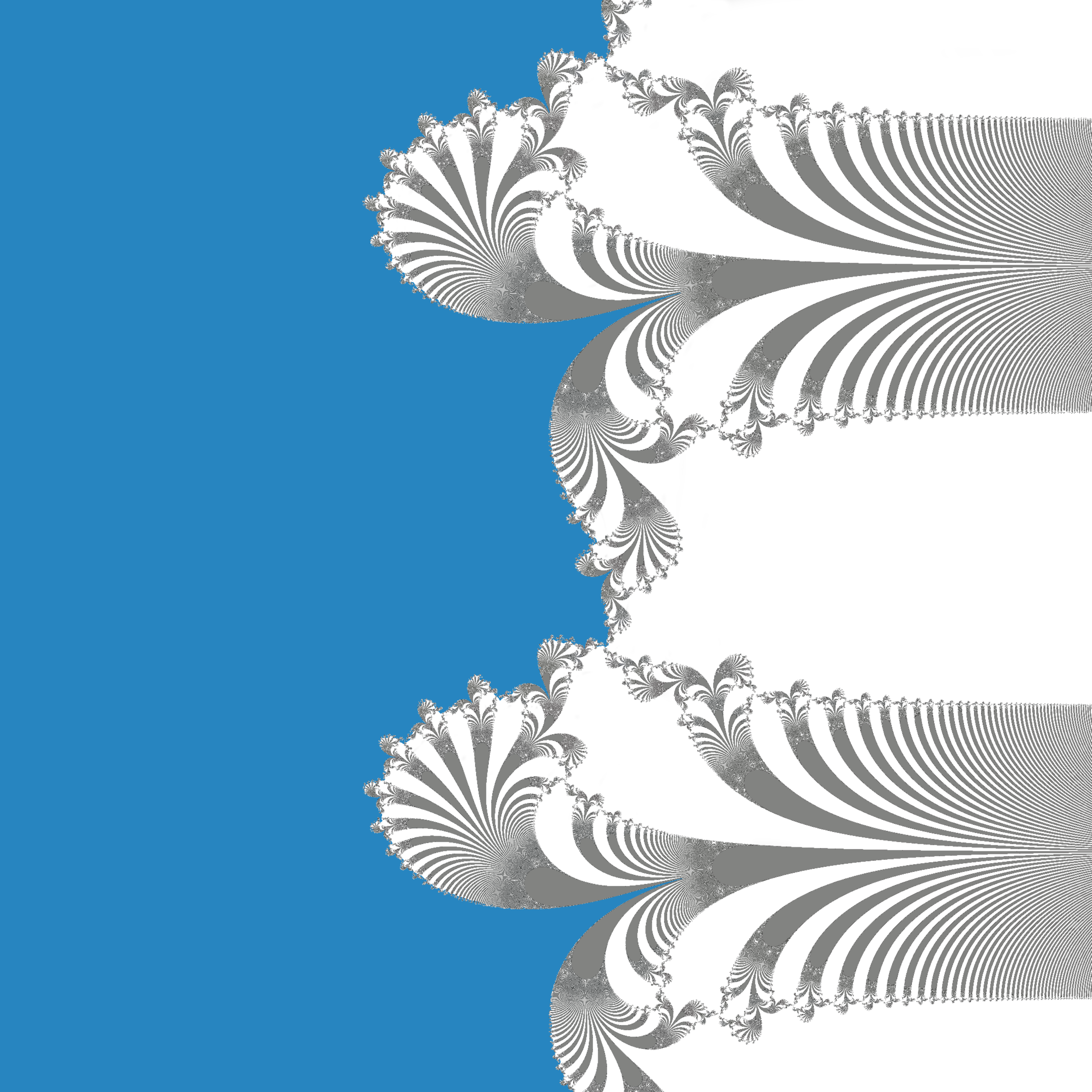} \:
\includegraphics[width=.315\textwidth]{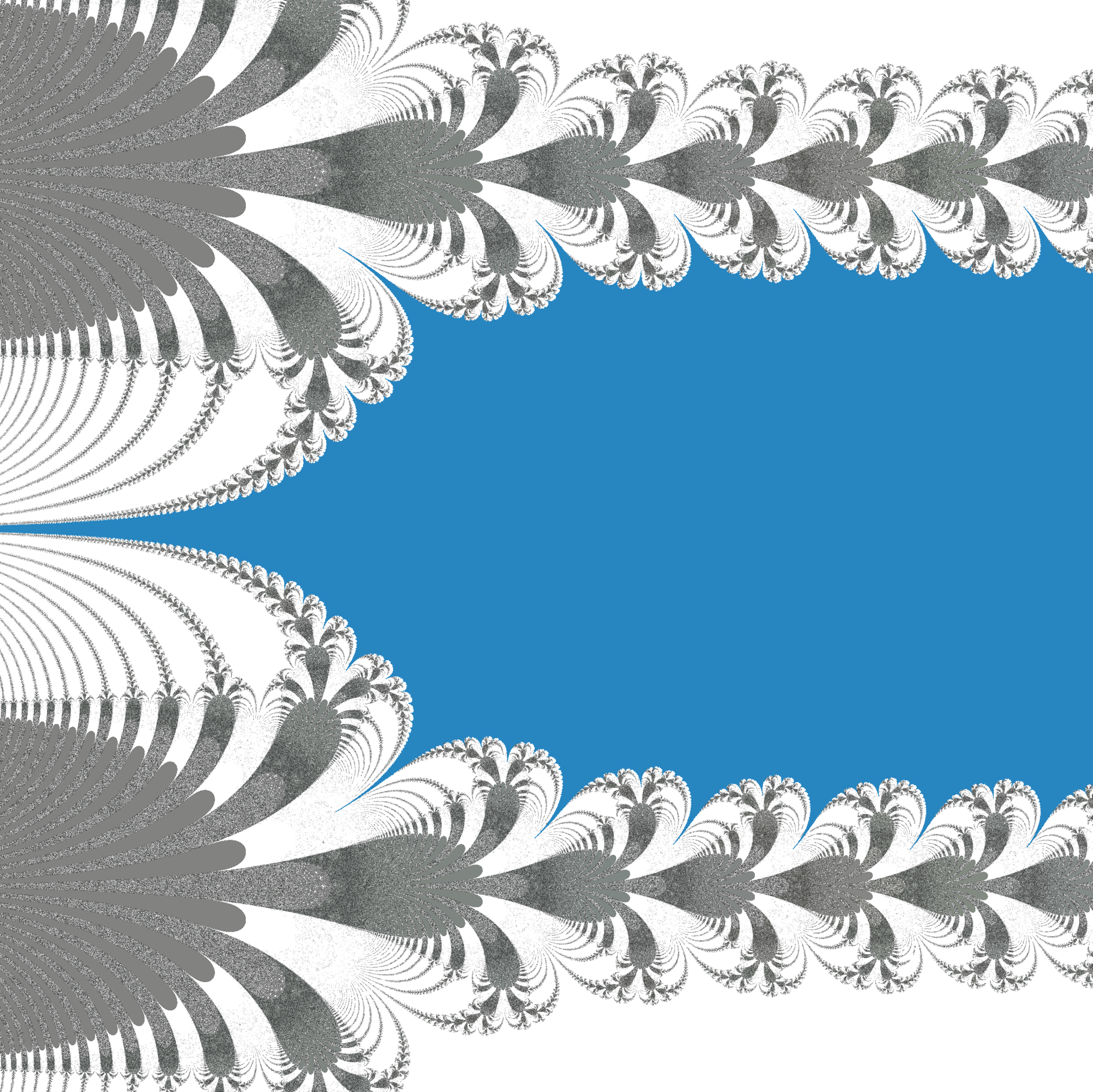} \:
\includegraphics[width=.315\textwidth]{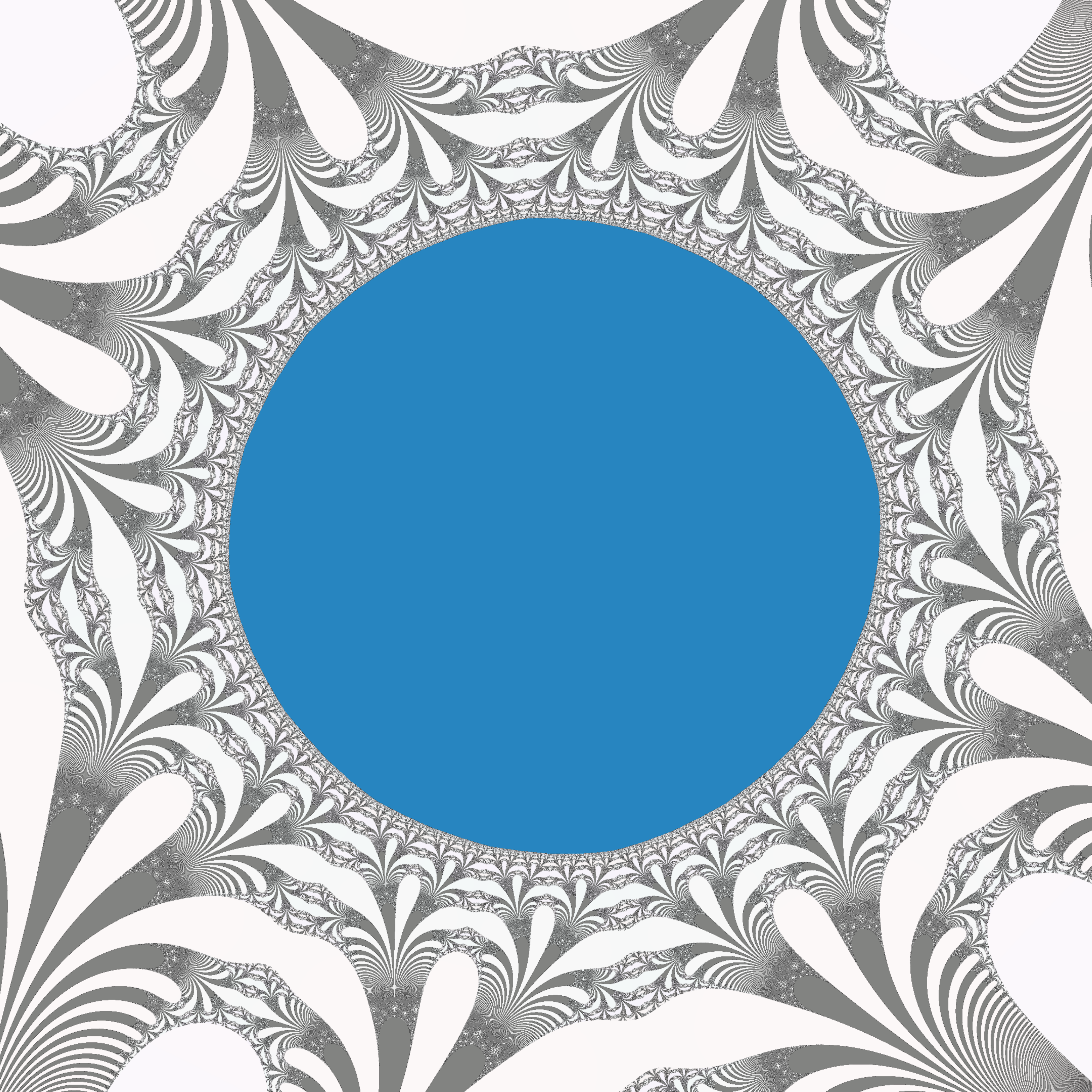} 
\caption{The three types of periodic components of attracting basins of entire transcendental maps. Left: an unbounded periodic component $U$ of period $3$ with $\deg f^3|_U = \infty$ for $f(z) = (2+\pi i) e^z$ (see \cite{bhattacharjee-devaney00}). Center: an unbounded invariant component $U$ with $\deg f|_U = 2$ for $f(z)=e^{-z}+z-1$ (see \cite{baker-dominguez,weinreich}). Right: a bounded invariant component $U$ with $\deg f|_U = 2$ for $f(z)= (-21+3i) z^2e^z$ (see \cite{garijo-jarque-morenorocha10}).}
\label{fig:three_types}
\end{figure}

In this paper we answer the above question in the affirmative for components of attracting and parabolic basins with an immediate component of the first type (for maps from class~$\mathcal B$, under a natural hyperbolicity condition) and for components of attracting basins with an immediate component of the third type (for arbitrary transcendental entire  maps). In fact, we show that the boundaries of the considered basins have hyperbolic dimension larger than~$1$. Recall that we write $\Sing(g)$ for the set of singular (critical or asymptotic) values of a holomorphic map $g$ defined on some domain in $\C$. Our main result is the following.

\begin{thmA}
Let $f\colon \C \to \C$ be a transcendental entire function from class~$\mathcal B$, and let $\zeta \in \C$ be an attracting periodic point of period $p\ge 1$. Suppose that for an immediate component $U$ of the attracting basin of $\zeta$, the degree of $f^p$ on $U$ is infinite and $\overline{\Sing(f^p|_U)}$ is a compact subset of~$U$. Then the boundary of each component of the basin of $\zeta$ has hyperbolic dimension larger than~$1$.
\end{thmA}

An analogous result for the parabolic basins is as follows.

\begin{thmB}
Let $f\colon \C \to \C$ be a transcendental entire function from class~$\mathcal B$, and let $\zeta \in \C$ be a parabolic periodic point of period $p\ge 1$. Suppose that for an immediate component $U$ of the parabolic basin of $\zeta$, the degree of $f^p$ on $U$ is infinite, $\overline{\Sing(f^p|_U)}$ is a compact subset of $U$, and $\zeta \notin \overline{\Sing(f^p)}$. 
Then the boundary of each component of the basin of $\zeta$ has hyperbolic dimension larger than~$1$.
\end{thmB}

\begin{remark}\label{rem:event_hyp}
In fact, the proofs of Theorems~A and~B show that in both cases, for every $r>0$ there exists an expanding conformal repeller $X \subset \{z \in \C: |z| > r\}$ of Hausdorff dimension larger than~$1$, contained in $\bd U$. Therefore, the \emph{eventual hyperbolic dimension} (in the sense of \cite{dezotti}) of the boundary of $U$ is at least~$1$.
\end{remark}

The proofs of Theorems~A and~B proceed by finding a suitable expanding conformal repeller $X \subset \partial U$ of Hausdorff dimension larger than~$1$. To this end, we follow a general idea of constructing such repellers in the Julia set, introduced in \cite{logtract}, which was later used and developed e.g.~in \cite{basexp,bergweiler-peter13,VolkerMayer,waterman20}. However, to ensure that the repeller $X$ is contained in $\partial U$, we need a careful modification of the construction, based on an analysis of topological and combinatorial structure of the considered basin.

\begin{remark}\label{rem:from_above}
In \cite{basexp} it was proved that for an immediate component $U$ of the basin of an attracting periodic point of period $p > 1$ of an exponential map, the set of points in the boundary of $U$ which escape to infinity under iteration of the map, has Hausdorff dimension equal to~$1$. In \cite{Berg-Ding}, Bergweiler and Ding generalized this result to some class of transcendental entire functions of finite order with a finite set of singularities.
\end{remark}

Our next result concerns components of attracting basins of transcendental entire functions with a bounded immediate component $U$. Since in this case the map is defined on a neighbourhood of the closure of $U$ in the Riemann sphere $\clC$, the situation is more similar to the one appearing for polynomial and rational maps. In this context, Zdunik proved in \cite{zdunik90,zdunik91} that for a polynomial $f$ with a simply connected basin $U$ of infinity, the following dichotomy holds: either $\partial U$ is an analytic curve, or $\dim_H \partial U >1$. This result was later extended to arbitrary simply connected components of attracting basins of rational functions by Przytycki \cite[Theorem~A]{przytycki06}. In \cite{hamilton95}, Hamilton proved that if the Julia set of a rational function $f$ is a Jordan curve in $\clC$, then either it is an analytic curve, or $\dim_H J(f) > 1$.

A classical result of Brolin \cite[Lemma~9.1]{brolin65} says that if $U$ is an invariant component of an attracting basin of a rational map $f$, and $\partial U$ is an analytic Jordan curve or an analytic arc, then $\partial U$ is, respectively, a circle or an arc of a circle in $\clC$. These two cases actually appear within the space of rational maps, and correspond, respectively, to $f$ being conformally conjugate to a finite Blaschke product $B$ with an attracting fixed point, or $f$ being semiconjugate to $B$ by the map $\phi(z)=z+1/z$ (which includes the case when $f$ is a Chebyshev polynomial and $B(z)=z^d$ for some integer $d \ge 2$), see \cite[Lemma~15.5 and Chapter~7]{milnor06}. 

In the case when $\bd U$ is a Jordan curve, the Brolin result was generalized by Azarina \cite{Azarina89}, who showed that in fact, if $\gamma \subset \C$ is an invariant analytic Jordan curve for an entire map $f$, then $f$ is conformally conjugate to a finite Blaschke product, or $\gamma$ is a level curve of an invariant Siegel disc of $f$.

In this paper, combining the Brolin--Azarina result with a theorem of Przytycki \cite[Theorem~A']{przytycki06}, which generalizes Zdunik's dichotomy to the case of so-called RB-domains (see Definition~\ref{defn:RB}), we show the following.

\begin{thmC}
Let $f\colon \C \to \C$ be a transcendental entire function, and let $\zeta \in \C$ be an attracting periodic point, such that an immediate component of the attracting basin of $\zeta$ is bounded. Then the boundary of each component of the basin of $\zeta$ has hyperbolic dimension larger than~$1$.
\end{thmC}

Note that the proof of Theorem~C, which uses different tools, is independent of the proofs of Theorems~A and~B.

In the context of transcendental entire maps, one can find few examples of Fatou components with smooth boundaries. In the above-mentioned paper \cite{bishop18}, Bishop constructed a transcendental entire function with a multiply connected wandering domain $U$, whose boundary consists of countably many $C^1$-Jordan curves. It is also known that for several families of transcendental entire functions there exist Siegel disks, whose boundaries are $C^\infty$-Jordan curves (see \cite{avila-buff-cheritat04,geyer08}). Recently, Boc Thaler \cite{bocthaler21} showed that every regular simply connected domain whose closure has no bounded complementary components, is a wandering domain of some transcendental entire function, which provides new examples of wandering domains with smooth boundaries. 

For transcendental entire functions, Baker and Weinreich \cite[Theorem~4]{baker-weinreich91} (see also \cite{kisaka-loc_con}) proved that if $U$ is an unbounded immediate component of an attracting basin, then its boundary in the Riemann sphere is not a curve. This together with Theorem~C immediately implies the following.

\begin{corD} Let $f$ be a transcendental entire function and let $U$ be a component of the basin of an attracting periodic point of $f$. Then the boundary of $U$ is not a smooth or rectifiable curve. 
\end{corD}

As an example, consider a bounded immediate component $U$ of the basin of an attracting fixed point of the map  $f(z)= (-21+3i) z^2e^z$, studied in \cite{garijo-jarque-morenorocha10} (see the right-hand side of Figure~\ref{fig:three_types}). Although $U$ appears to have a smooth boundary, Corollary~D implies that $\bd U$ is neither smooth nor rectifiable.

\begin{remark}\label{rem:mero} Note that Corollary~D is not true for transcendental meromorphic maps. Indeed, for $f(z) = \lambda\tan z$, $\lambda > 1$, the upper and lower half-planes are completely invariant components of the  basins of attracting fixed points, and their common boundary $\R \cup \{\infty\}$, equal to the Julia set of $f$, is an analytic Jordan curve in the Riemann sphere. Concerning other types of Fatou components, 
Hamilton \cite{hamilton-rect} proved that there exist invariant Baker domains $U$ (that is, Fatou components where the iterates of the map tend to infinity) of meromorphic maps, such that $\bd U$ is a non-smooth rectifiable Jordan curve.
\end{remark}

The structure of the paper is as follows. Section~\ref{sec:prelim} contains the preliminaries.  The proof of Theorem~A, divided into several parts, is included in Sections~\ref{sec:top struct}--\ref{sec:est_der}. Since the proof is rather involved, before going into details, in Section~\ref{sec:summ_proofA} we present its short description. In Section~\ref{sec:parab} we explain how to adapt the proof of Theorem~A to the case of parabolic basins, which enables showing Theorem~B. Finally, in Section~\ref{sec:bounded} we prove Theorem~C, with a proof of the Brolin--Azarina lemma for attracting basins of transcendental entire maps included in the appendix.

\section*{Acknowledgements}
We are grateful to Lasse Rempe for informing us about Azarina's work \cite{Azarina89}, and to the referees of this paper for useful remarks and comments.
  
\section{Preliminaries}\label{sec:prelim}

\subsection{General notation}\label{subsec:notat} By $\overline{A}$ and $\bd A$ we denote, respectively, the closure and boundary (in $\C$) of a set $A \subset \C$. The open disc of radius $r>0$, centered at $z \in \C$, is denoted by $\D(z,r)$. The open unit disc is denoted by $\D$. For $x \in \R$ we write
\[
\mathbb H_{>x} = \{z\in\C: \Re(z) > x\}
\]
and define $\mathbb H_{<x}$, $\mathbb H_{\ge x}$, $\mathbb H_{\le x}$ analogously. We also set
\[
\Z_{\ge \sigma} = \{s \in \Z: s \ge \sigma\}
\]
for $\sigma \in \Z$. We denote by $\exp$ the exponential map $z \mapsto e^z$ on $\C$.

\subsection{Hyperbolic metric} \label{subsec:hyp} For a simply connected domain $V \subset \C$
we denote by $d\rho_V$ the \emph{hyperbolic metric} on $V$, which is a conformal Riemannian metric given by $d\rho_V = \frac{2|\varphi'(z)|}{1 - |\varphi(z)|^2}|dz|$, where $\varphi\colon V \to \D$ is a Riemann map. The Schwarz--Pick lemma (see e.g.~\cite[Theorem~I.4.1]{carlesongamelin}), asserts that for a holomorphic map $g\colon V \to V$, one has $\rho_V(g(z_1), g(z_2)) \le \rho_V(z_1, z_2)$, $z_1, z_2 \in V$, for the \emph{hyperbolic distance} $\rho_V$ determined by the metric $d\rho_V$, where the inequality is strict unless $g$ is univalent.

\subsection{Distortion of univalent maps and Koebe's theorems} \label{subsec:koebe} The \emph{distortion} of a univalent map $g$ on a set $A \subset \C$ is defined as $\sup_{z_1, z_2 \in A} \frac{|g'(z_1)|}{|g'(z_2)|}$. We recall the classical Koebe distortion and one-quarter theorems (see e.g.~\cite[Theorems~I.1.3 and~I.1.6]{carlesongamelin}).

\begin{KDT}
Let $g\colon \D(z_0,r) \to \C$ be a univalent holomorphic map, for some $z_0 \in\C$ and $r > 0$. Then for every $z \in \overline{\D(z_0,\lambda r)}$, where  $0 < \lambda < 1$, 
\[
\frac{1 - \lambda}{(1 + \lambda)^3} \leq \frac{|g'(z)|}{|g'(z_0)|} \leq \frac{1 + \lambda}{(1 - \lambda)^3}.
\]
\end{KDT}

\begin{KOT}
Let $g\colon \D(z_0,r) \to \C$ be a univalent holomorphic map, for some $z_0 \in\C$ and $r > 0$. Then $\D(g(z_0), |g'(z_0)|r/4) \subset g(\D(z_0,r))$.
\end{KOT}

\subsection{Conformal repellers and Bowen's formula}\label{subsec:bowen}
Let $g\colon V \to \C$ be a holomorphic map on an open set $V \subset \C$. An \emph{expanding conformal repeller} $\Lambda \subset V$ is a compact set, such that $g(\Lambda) \subset \Lambda$ and $|(g^k)'| > 1$ on $\Lambda$ for some $k\in\mathbb{N}$. The classical \emph{Bowen's formula} (see e.g.~\cite[Corollary~9.1.7]{PUbook}) asserts that the Hausdorff dimension of an expanding conformal repeller $\Lambda$ is equal to the unique zero of the (continuous and strictly decreasing) \emph{topological pressure function} $t \mapsto P(g|_\Lambda, t)$, $t>0$, where
\[
P(g|_\Lambda, t) = \lim_{n\to\infty} \frac{1}{n} \ln \sum_{w \in g^{-n}(z) \cap \Lambda} |(g^n)'(w)|^{-t}
\]
for $z \in \Lambda$. 

\subsection{Useful facts on entire functions}\label{subsec:entire}

We use Liouville's theorem, which states that every bounded entire function is constant, and Picard's theorem, which asserts that a transcendental entire function attains every complex value, with at most one exception, in any neighbourhood of infinity. 

Finally, let us recall a useful fact on the degree of an entire function on preimages of simply connected domains.

\begin{proposition}[{\cite[Proposition 2.8]{bergweiler-fagella-rempegillen15}}]\label{prop:swiss}
Let $f$ be a non-constant entire function, let $U \subset \C$ be a simply connected domain and let $V$ be a component of $f^{-1}(U)$. Then either $f$ maps $V$ onto $U$ as a proper map with a finite degree, or $V$ is unbounded and the set $f^{-1}(z) \cap V$ is infinite for every $z \in U$ with one possible exception. In the second case, $V$ contains an asymptotic curve corresponding to an asymptotic
value in $U$, or $V$ contains infinitely many critical points of $f$.
\end{proposition}

Note that periodic and preperiodic Fatou components of transcendental entire maps are always simply connected (see \cite{baker75}).

\section{Summary of the proof of Theorem~A}\label{sec:summ_proofA}
As the proof of Theorem~A is rather involved, we present here its short description, introducing notation and highlighting the main ideas and steps of the proof. Note first that since the hyperbolic dimension of the boundary of a Fatou component is preserved by branches of $f^{-1}$, it is enough to show that the boundary of the immediate component $U$ has hyperbolic dimension larger than $1$. We can also assume $p=1$, so that $U$ is invariant.

The whole proof, divided into several parts, spans Sections~\ref{sec:top struct}--\ref{sec:est_der}. In Section~\ref{sec:top struct}  we describe the topological structure of a (simply connected and unbounded) attracting component $U$ with $\deg f|_U = \infty$. In Lemma~\ref{lem:D} it is shown that there exists a simply connected domain $D \subset U$, such that $E = f^{-1}(D) \cap U$ is an unbounded simply connected domain containing $\overline D$, and each component of $U \setminus \overline{E}$ is a simply connected domain, mapped by $f$ onto $U \setminus \overline{D}$ as an infinite degree covering.

Section~\ref{sec:lifts} concerns two kinds of logarithmic lifts of the map $f$. First, in Subsection~\ref{subsec:U-E} we note that due to Lemma~\ref{lem:D}, the map $f$ on $U \setminus E$ (assuming $\zeta = 0$) can be lifted by the exponential map $\exp$ to the map
\[
\Phi \colon \exp^{-1}(U \setminus E) \to \exp^{-1}(U \setminus D),
\]
such that the inverse branches of $\Phi$, denoted by $G_s$, $s \in \Z$,  form an iterated function system on $\exp^{-1}(U \setminus D)$. 
For a chosen component $\mathcal V$ of $U \setminus \overline{E}$ and $\omega = \bd \mathcal V \cap U$, we define a system of curves in $\exp^{-1}(U)$, consisting of a logarithmic lift $\Gamma$ of $\bd D$, logarithmic lifts $\Gamma_s$, $s \in \Z$, of $\omega$, and some chosen curves $\gamma_s$ connecting $\Gamma$ to $\Gamma_s$. In Lemma~\ref{lem:GammainL} we show that the union of all these curves, together with their images under successive iterates of the branches $G_s$, is a connected tree-like subset of $\exp^{-1}(U)$. 

In Subsection~\ref{subsec:near_inf} we consider another logarithmic lift $F$ of the map $f$, defined on the union of lifted tracts of $f$ over infinity, which exists due to the assumption $f \in \mathcal B$. The inverse branches of $F$ have good contracting properties. Note that due to the possible existence of the critical points of $f$ with critical values that lie outside of $U$, the inverse branches of the two lifts $\Phi$ and $F$ need not coincide on the intersection of their domains of definition. 

A general idea of the construction is to find a suitable iterated function system composed of inverse branches of some iterate of $F$, defined on a compact set $K$ (the limit set of such system can be projected by $\exp$ to create an appropriate expanding conformal repeller for $f$). However, to ensure that the limit set of this system is contained in $\exp^{-1}(\bd U)$, we can use only the branches which map points from $\exp^{-1}(U)$ to points from $\exp^{-1}(U)$. To this aim, we should study the relations between the lifts $\Phi$ and $F$, which is done in Section~\ref{sec:rel}. First, in Lemma~\ref{lem:alpha} we define some special points $v_s \in \Gamma_s \subset \exp^{-1}(U)$, located at the boundary of a simply connected domain $T$ with unbounded positive real part, where the inverse branches of $F$ are defined (see \eqref{eq:TinH}). We set 
\[
v_{s_0, \ldots, s_n} = G_{s_0} \circ \cdots \circ G_{s_{n-1}}(v_{s_n}) \in \exp^{-1}(U).
\]
In Lemma~\ref{lem:hats} and Corollary~\ref{cor:vinH} we show that $v_{s_0, \ldots, s_n} \in T$ provided $s_1 \ge \sigma$ for a sufficiently large fixed  number $\sigma \in \Z$. The most important part of this section is given by Lemma~\ref{lem:k(s)}, which asserts that for $s_0,s_1 \in \Z$ there exist an inverse branch $H_{s_0}^{(s_1)}$ of $F$ and a number $k(s_1) \in \Z$, such that
\[
H_{s_0}^{(s_1)}(v_{k(s_1), s_2,  \ldots, s_n}) = v_{s_0, \ldots, s_n}
\]
for $s_2 \ge \sigma$. These are the `good' branches, which will be used to construct a suitable iterated function system on a compact set $K \subset T$.

In Sections~\ref{sec:prelim_est}--\ref{sec:constr} we present a precise definition of the iterated function system. We define $K$ to be a rectangle in $T$ of width $2$ and sufficiently large height, such that $K \subset \mathbb H_{> b}$ for a large $b>0$, while the set $T_\sigma = H_{\sigma}^{(\sigma)}(T)$ `passes through' $K$ from the left to the right (see \eqref{eq:throughK}). To define a suitable iterated function system on $K$, an initial idea, originating from \cite{logtract}, is to use the maps 
\[
\psi_{j, l} = H_{k(\sigma)}^{(j)} \circ H_{k(j)}^{(l)} \circ H_{k(l)}^{(\sigma)}, \qquad j,l \ge \sigma,
\]
which for each $j$ create a sequence $(\psi_{j, l}(K))_{l = \sigma}^\infty$ of sets with real parts converging to $+\infty$ (see Subsection~\ref{subsec:psi}). However, since the set of pairs $(k(j), k(l))$, $j,l \ge \sigma$, can be very sparse within all integer pairs $\Z_{\ge \sigma} \times \Z_{\ge \sigma}$, the sets $\psi_{j, l}(K)$ may be located far to the right, at large distances to each other. To overcome this problem, we consider the maps 
\[
\phi_{j,l} =  H_\sigma^m\circ\psi_{j, l}, \qquad j,l \ge \sigma
\]
for a large positive integer $m$, where $H_\sigma = H_{k(\sigma)}^{(\sigma)}$. Applying a high iterate of $H_\sigma$ enables us to `move' the sets $\psi_{j, l}(K)$ back to the initial rectangle $K$. 

The precise definition of the rectangle $K$ is presented in Subsection~\ref{subsec:K}, while in Subsection~\ref{subsec:psi} we provide estimates for the maps $\psi_{j, l}$. In particular, Lemma~\ref{lem:a-u} shows that for a fixed $u \in K$, the points $\psi_{j, l}(u), v_{k(\sigma),j,l}, v_{k(\sigma),j,l+1}$ lie at a uniformly bounded distance to each other, while 
\begin{equation}\label{eq:summ-psi}
|\psi_{j, l}'(u)| \ge c |v_{k(\sigma),j,l} - v_{k(\sigma),j,l+1}| 
\end{equation}
for a constant $c>0$.

In Section~\ref{sec:constr} we prove an important Lemma~\ref{lem:seq}, which asserts that for a large $R>0$, one can find a set $\mathcal J \subset \Z_{\ge \sigma}$ of an arbitrarily large cardinality $M$, some `initial' indices $l = \ell_j\in \Z_{\ge \sigma}$ for $j \in \mathcal J$, and a sufficiently large integer $m$, such that the sequences 
\[
(H_\sigma^m (\D(\psi_{j,l}(u),R)))_{l = \ell_j}^\infty, \qquad j \in \mathcal J
\]
are well-defined and consist of $\varepsilon$-small subsets of $T_\sigma$, where the `initial' ones
\[
H_\sigma^m (\D(\psi_{j,\ell_j}(u),R))
\]
are located in $\mathbb H_{<b}$, which is to the left of the rectangle $K$. Noting that $\psi_{j,l}(K) \subset \D(\psi_{j,l}(u),R)$ for sufficiently large $R$, we can define the appropriate iterated function system on $K$ as
\[
\{\phi_{j,l}|_K\}_{(j,l) \in \mathcal K},
\]
where $\mathcal K$ is the set of pairs $(j,l)$, $j \in \mathcal J$, $l \ge \ell_j$, with $\phi_{j,l}(K) \subset K$. By definition, the attractor $\Lambda$ of the system $\{\phi_{j,l}|_K\}_{(j,l) \in \mathcal K}$ is an expanding conformal repeller for $F^{m+3}$. 

In the concluding Section~\ref{sec:est_der},
we show that the Hausdorff dimension of $\Lambda$ is larger
than~$1$, estimating the derivatives of the branches $\phi_{j,l}$ on $K$ and applying tools coming from thermodynamic formalism. First, in Lemma~\ref{lem:der-phi} we note that as by Lemma~\ref{lem:a-u} we have $v_{k(\sigma),j,l} \in \D(\psi_{j,l}(u),R))$, Lemma~\ref{lem:seq} implies that the points 
\[
w_{j,l} = H^m_\sigma(v_{k(\sigma),j,l}) = v_{\scriptsize k(\sigma), \underbrace{\sigma, \ldots, \sigma}_{m \text{ times}}, j, l} \in T_\sigma \cap \exp^{-1}(U), \qquad j \in \mathcal J, \; l \in \Z_{\ge \ell_{j}}
\]
are located $\varepsilon$-close to the sets $\phi_{j,l}(K)$ and (by \eqref{eq:summ-psi}) satisfy
\begin{equation}\label{eq:summ-phi}
|\phi_{j, l}'(u)| \ge c |w_{j,l} - w_{j,l+1}|
\end{equation}
for a constant $c>0$ (see Lemma~\ref{lem:der-phi}). Furthermore, using 
Lemma~\ref{lem:seq}, we check 
\[
\Re(w_{j,\ell_j}) < b, \qquad |w_{j,l} - w_{j,l+1}| < \varepsilon, \qquad \lim_{l \to \infty}\Re(w_{j,\ell_j}) = +\infty.
\]
Consequently, for each $j \in \mathcal J$, the sequence $(w_{j,l})_{l = \ell_{\sigma}}^\infty$ `passes through' $K$, from the left to the right, with subsequent points located $\varepsilon$-close to each other. As $K$ has width $2$, this implies  
\begin{equation}\label{eq:summ-w}
\sum_{l:(j,l) \in \mathcal K}|w_{j,l} - w_{j,l+1}| > 1
\end{equation}
for $j \in \mathcal J$ (see Lemma~\ref{lem:sum-v}). By \eqref{eq:summ-phi} and \eqref{eq:summ-w}, we obtain
\[
\sum_{(j,l) \in \mathcal K}|\phi_{j, l}'(u)| > cM .
\]
Since $M$ can be chosen to be arbitrarily large, this implies that the topological pressure $P(F^{m+3}|_\Lambda, t)$ is positive for $t=1$, which in turn yields $\dim_H \Lambda > 1$. On the other hand, as $w_{j,l} \in \exp^{-1}(U)$, the points in $\Lambda$ are approximated by points from $\exp^{-1}(U)$ (see \eqref{eq:z-v<}), which implies $\Lambda \subset \exp^{-1}(\bd U)$. Defining $X = \bigcup_{n=0}^{m+2} f^n(\exp(\Lambda))$, we show the existence of an expanding conformal repeller $X \subset \bd U$ for $f$, of Hausdorff dimension larger than~$1$, which ends the proof of Theorem~A.

\section{Topological structure of attracting components of infinite degree}\label{sec:top struct}

As indicated at the beginning of Section~\ref{sec:summ_proofA}, to prove Theorem~A it is enough to estimate the dimension of the boundary of the immediate component $U$ satisfying the conditions of the theorem. Since $f \in \mathcal B$ implies $f^p \in \mathcal B$, we can also assume 
\[
p=1,
\]
so that $U$ is invariant and $\zeta$ is an attracting fixed point. Recall that $U$ is simply connected as a periodic Fatou component of an transcendental entire map. In this section we describe the topological structure of $U$, proving the following lemma.

\begin{lemma}\label{lem:D} There exists a bounded simply connected domain $D \subset U$ with a Jordan boundary, such that $\overline{D} \subset U$, $\overline{\Sing(f|_U)} \subset D$ and $\overline{f(D)} \subset D$. Furthermore, $E = f^{-1}(D) \cap U$ is an unbounded simply connected domain containing $\overline D$.  Every component $\mathcal V$ of $U \setminus \overline{E}$ is a simply connected domain, mapped by $f$ onto $U \setminus \overline{D}$ as an infinite degree covering, such that $\bd \mathcal V \cap U$ is homeomorphic to the real line, and is mapped by $f$ onto $\bd D$ as an infinite degree covering.
\end{lemma}
\begin{proof} Since by assumption, $\overline{\Sing(f|_U)}$ is a compact subset of $U$, we can choose a sufficiently large $r>0$, such that for
\[
D = \{z \in U: \rho_U(z, \zeta) < r\}
\]
we have $\overline{\Sing(f|_U)} \subset D$, where $\rho_U$ is the hyperbolic metric on $U$. Obviously, $\overline{D} \subset U$. As $U$ is simply connected, $\overline{D}$ is a homeomorphic image of a closed disc under a Riemann map, so $D$ is a simply connected domain and $\bd D$ is a Jordan curve.  Moreover, $\overline{f(D)} \subset D$ by the Schwarz--Pick lemma. See Figure~\ref{fig:D}.

\begin{figure}[ht!]
\centerline{\includegraphics[width=0.99\textwidth]{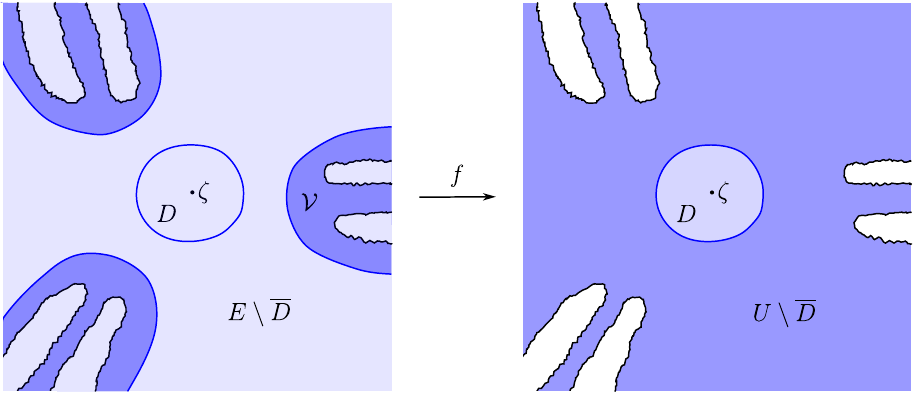}}
\caption{Topological structure of the component $U$.}\label{fig:D}
\end{figure}

As $U$ is simply connected and $\overline{\Sing(f|_U)} \subset D \subset \overline{D} \subset U$, the set $U \setminus \overline{D}$ is homeomorphic to an annulus, and $U \setminus D$ is disjoint from $\overline{\Sing(f|_U)}$. Hence, for each component $\mathcal V$ of $f^{-1}(U \setminus \overline{D}) \cap U$, the two following cases are possible: either $\mathcal V$ is homeomorphic to an annulus, and is mapped under $f$ onto $U \setminus \overline{D}$ as a covering of degree $d$ for some $d \in \N$, such that $\bd\mathcal V \cap U$ is a Jordan curve mapped onto $\bd D$ as a degree $d$-covering, or $\mathcal V$ is an unbounded simply connected domain, which is mapped by $f$ onto $U \setminus \overline{D}$ as an infinite degree covering, such that $\bd\mathcal V \cap U$ is a curve homeomorphic to the real line, with two `ends' tending to infinity, mapped onto $\bd D$ as an infinite degree covering. Suppose the first case holds. Then, denoting by $\mathcal C$ the bounded component of the complement of the Jordan curve $\bd\mathcal V \cap U$, we have $\overline{\mathcal C} \cup \mathcal V = U$ and $f(\overline{\mathcal C}) =\overline{D}$ by the maximum principle, so $\mathcal V = f^{-1}(U \setminus \overline{D}) \cap U$, which contradicts the assumption $\deg f|_U = \infty$. Therefore, for every component $\mathcal V$ of $f^{-1}(U \setminus \overline{D}) \cap U$ the second case holds. This implies that the set $E = f^{-1}(D) \cap U$ is an unbounded simply connected domain, while the components of $U \setminus \overline{E}$ coincide with the components of $f^{-1}(U \setminus \overline{D}) \cap U$. As $\overline{f(D)} \subset D$, we have $\overline{D} \subset E$. See Figure~\ref{fig:D}. This ends the proof of the lemma. 
\end{proof}

\section{Two logarithmic lifts of $f$}\label{sec:lifts}
In the following two subsections we describe, respectively, two logarithmic lifts of the map $f$, defined on some subsets of $U$. For simplicity, we may assume 
\[
\zeta = 0
\]
after a coordinate change by a translation. 

\subsection{\boldmath Logarithmic lift of $f$ on $U \setminus E$}\label{subsec:U-E}
Consider the sets $D$ and $E$ from Lemma~\ref{lem:D}. Recall that Lemma~\ref{lem:D} shows that every component of $U \setminus E$ is a simply connected set, mapped by $f$ onto $U \setminus D$ as an infinite degree covering. Hence, the map $f$ on $U \setminus E$ can be lifted by the exponential map $\exp$ to a map 
\[
\Phi \colon \exp^{-1}(U \setminus E) \to \exp^{-1}(U \setminus D),
\]
such that 
\[
\exp \circ \Phi = f \circ \exp
\]
and
\[
\Phi(z + 2\pi i) = \Phi(z) \qquad \text{for } \; z \in \exp^{-1}(U \setminus E).
\]
Let 
\[
V = \exp^{-1} (U \setminus \overline{D}), \qquad \Gamma = \exp^{-1} (\bd D).
\]
Then $V$ is a simply connected domain with unbounded positive real part, $\Gamma$ is a component of $\bd V$, homeomorphic to the real line, and 
\[
V = V + 2 \pi i s, \qquad \Gamma = \Gamma + 2 \pi i s, \qquad s \in \Z.
\]

Choose a component $\mathcal V$ of $U \setminus \overline{E}$ and let 
\[
\omega = \bd\mathcal V \cap U.
\]
By Lemma~\ref{lem:D}, the set $\omega$ is a component of $\bd E$, which is an unbounded curve homeomorphic to the real line, mapped by $f$ onto $\bd D$ as an infinite degree covering. Moreover,
\[
\exp^{-1}(\mathcal V) = \bigcup_{s \in \Z} V_s \subset V,
\]
where $V_s$, $s \in \Z$, are pairwise disjoint simply connected domains with unbounded positive real part, such that $V_s = V_0 + 2\pi i s$. Furthermore, 
\[
\Gamma_s = \exp^{-1}(\omega) \cap \bd V_s = \bd V_s \cap V, \qquad s \in \Z,
\]
is a curve contained in $V$, homeomorphic to the real line, and its two `ends' have real part tending to infinity. Note that $\Phi$ maps univalently $V_s$ onto $V$, while $\Gamma_s$ is mapped homeomorphically onto $\Gamma$.
Let
\[
G_s = (\Phi|_{V_s \cup \Gamma_s})^{-1} \colon V \cup \Gamma \to V_s \cup \Gamma_s, \qquad s \in \Z. 
\]
Then $G_s$ is a homeomorphism, univalent on $V$, and
\[
G_s = G_0 + 2\pi is.
\]

For $s_0, \ldots, s_n \in \Z$, $n \ge 1$, define inductively
\[
V_{s_0, \ldots, s_n} = G_{s_0} \circ \cdots \circ G_{s_n}(V), \qquad \Gamma_{s_0, \ldots, s_n} = G_{s_0} \circ \cdots \circ G_{s_n}(\Gamma).
\]
Then $V_{s_0, \ldots, s_n}$ is a simply connected domain, $\Gamma_{s_0, \ldots, s_n} = \bd V_{s_0, \ldots, s_n} \cap V$ and
\begin{equation}\label{eq:V} V_s \cup \Gamma_s \subset V, \qquad
V_{s_0, \ldots, s_n} \cup \Gamma_{s_0, \ldots, s_n} \subset V_{s_0, \ldots, s_{n-1}} \subset V \qquad \text{for }\;s, s_0, \ldots, s_n \in \Z, \; n \ge 1.
\end{equation}
Parameterize $\Gamma$ by $t \in \R$, such that 
\[
\Gamma(t + s) = \Gamma(t) + 2 \pi i s \qquad \text{for } \; t \in \R, \; s \in \Z
\]
and transfer this parameterization to $\Gamma_{s_0, \ldots, s_n}$ by $G_{s_0} \circ \cdots \circ G_{s_n}$. 
Let 
\[
a_s = \Gamma(s), \qquad a_{s_0, \ldots, s_n} = G_{s_0} \circ \cdots \circ G_{s_{n-1}}(a_{s_n}), \qquad s, s_0, \ldots, s_n \in \Z, \;n \ge 1
\]
and define
\[
\gamma_s \colon [0,1] \to \C, \qquad s \in \Z, 
\]
to be pairwise disjoint curves homeomorphic to $[0,1]$, such that $\gamma_s(0) = a_s$, $\gamma_s(1) = a_{s,0}$, 
\[
\gamma_s((0,1)) \subset V \setminus \bigcup_{s' \in \Z} \Gamma_{s'}
\]
and $\gamma_s = \gamma_0 + 2\pi i s$. 
Let
\[
\gamma_{s_0, \ldots, s_n} = G_{s_0} \circ \cdots \circ G_{s_{n-1}}\circ \gamma_{s_n}, \qquad s_0, \ldots, s_n \in \Z, \; n\ge 1.
\]
See Figure~\ref{fig:G}.

\begin{figure}[ht!]
\centerline{\includegraphics[width=0.99\textwidth]{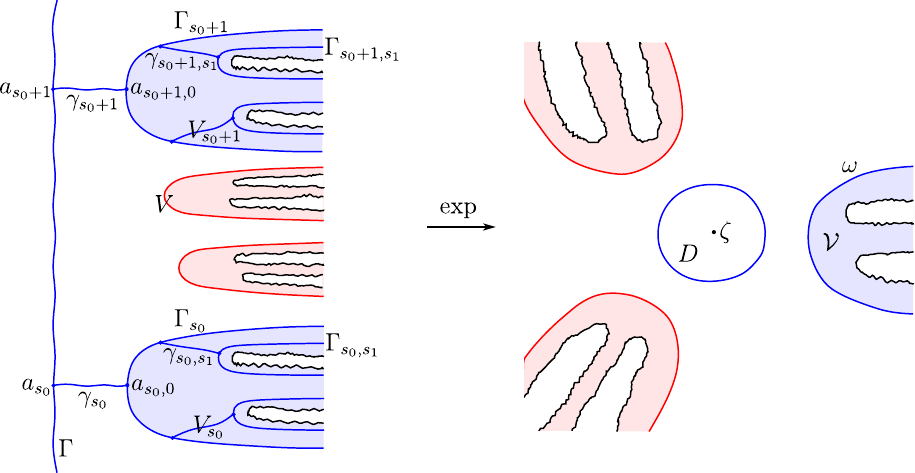}}
\caption{Logarithmic lift on $U \setminus E$.}\label{fig:G}
\end{figure}

The following lemma is a direct consequence of the properties of the curves $\gamma_{s_0, \ldots, s_n}$, $\Gamma_{s_0, \ldots, s_k}$, listed above.

\begin{lemma}\label{lem:GammainL} 
For every $s_0 \in \Z$, the set
\[
(\gamma_{s_0} \setminus \{a_{s_0}\}) \cup \Gamma_{s_0} \cup 
\bigcup_{k=1}^\infty\;\bigcup_{s_1, \ldots, s_k \in \Z} \gamma_{s_0, \ldots, s_k} \cup \Gamma_{s_0, \ldots, s_k}
\]
is a connected subset of $V$. 

For every $s_0, \ldots, s_n \in \Z$, $n \ge 1$, the set
\[
(\gamma_{s_0, \ldots, s_n}  \setminus \{a_{s_0, \ldots, s_n}\})\cup \Gamma_{s_0, \ldots, s_n} \cup \bigcup_{k=1}^\infty\;\bigcup_{s_{n+1}, \ldots, s_{n+k} \in \Z} \gamma_{s_0, \ldots, s_{n+k}} \cup \Gamma_{s_0, \ldots, s_{n+k}}
\]
is a connected subset of $V_{s_0,\ldots, s_{n-1}}$.
\end{lemma}

\subsection{\boldmath Logarithmic lift of $f$ near $\infty$}\label{subsec:near_inf}
Recall that for $x \in \R$ we write
\[
\mathbb H_{>x} = \{z\in\C: \Re(z) > x\},
\]
while $\mathbb H_{<x}$, $\mathbb H_{\ge x}$, $\mathbb H_{\le x}$ are defined analogously. 

Since $f \in \mathcal B$, we can fix a sufficiently large number $a>0$, such that $\overline{\Sing(f)}\subset \D(0, e^a)$ and consequently, every component $\mathcal W$ of $f^{-1}(\C \setminus \overline{\D(0, e^a)})$, called a \emph{tract of $f$ over infinity}, is a simply connected unbounded domain with the boundary homeomorphic to the real line with two `ends' tending to infinity, such that $f$ maps $\mathcal W$ onto $\C \setminus \overline{\D(0, e^a)}$ as an infinite degree covering, with a continuous extension to $\overline{\mathcal W}$. Hence, the map $f$ on $f^{-1}(\C \setminus \D(0, e^a))$ can be lifted by $\exp$ to a map 
\[
F\colon \exp^{-1}(f^{-1}(\C \setminus \D(0, e^a))) \to \mathbb H_{\ge a},
\]
such that
\[
\exp\circ F = f\circ \exp
\]
and
\[
F(z) = F(z +2\pi i)\qquad \text{for } \; z \in \exp^{-1}(f^{-1}(\C \setminus \D(0, e^a))).
\]
Moreover, for every tract $\mathcal W$ of $f$ over infinity, we have
\[
\exp^{-1}(\mathcal W) = \bigcup_{s \in \Z} W_s,
\]
where $W_s$, $s \in \Z$, called \emph{lifted tracts of $f$ over infinity}, are pairwise disjoint simply connected unbounded domains,
\[
W_s = W_0 + 2\pi i s,
\]
and $\bd W_s$ is homeomorphic to the real line, such that its two `ends' have real part tending to infinity. The map $F$ provides a homeomorphism from $\overline{W_s}$ onto $\mathbb H_{\ge a}$, univalent on $W_s$. For details, see e.g.~\cite{eremenko-lyubich92}.

\section{Relations between two logarithmic lifts}\label{sec:rel}

In this section we describe relations between the two logarithmic lifts $\Phi$ and $F$ of $f$, defined, respectively, in Subsections~\ref{subsec:U-E} and~\ref{subsec:near_inf}.
Note first that by the definition of $\Phi$ and $F$,
\begin{equation}\label{eq:Phi-F}
\exp\circ \Phi(z) = \exp \circ F(z) \qquad \text{for }\; z \in  F^{-1}(\mathbb H_{\ge a}) \cap \bigcup_{s\in \Z} (V_s \cup \Gamma_s).
\end{equation}
Furthermore, increasing $a$ if necessary, we can assume
\begin{equation}\label{eq:Gamma<a}
\Gamma \subset \mathbb H_{<a}.
\end{equation}

We start by introducing some basic objects, used in the further considerations.

\begin{lemma}[{\bf\boldmath Existence of the curve $\alpha$ and points $v_s$}{}]\label{lem:alpha} There exists a curve $\alpha\colon \R \to \C$ homeomorphic to the real line, such that
\begin{enumerate}[$($a$)$]
\item $\alpha(t + s) = \alpha(t) + 2 \pi is$ for $t \in \R$, $s \in \Z$,
\item $\alpha \subset \mathbb H_{> 2a}$,
\item for every $s \in \Z$, the curve $\alpha$ intersects $\Gamma_s([0, +\infty))$ at a unique point $v_{s}$.
\end{enumerate}
\end{lemma}
\begin{proof} Since $\exp(\Gamma_s([0, +\infty))) \subset \omega$ is a curve homeomorphic to $[0,+\infty)$, which tends to infinity, we can construct a Jordan curve $\mathcal A \subset \C \setminus \overline{\D(0, e^{2a})}$, intersecting $\exp(\Gamma_s([0, +\infty)))$ at a unique point. Lifting $\mathcal A$ by $\exp$ provides a suitable curve $\alpha$. 
\end{proof}

Consider the curve $\alpha$ and points $v_s$, $s \in \Z$, from Lemma~\ref{lem:alpha}. Note that $v_s \in V$ and
\[
v_s = v_0 + 2\pi i s.
\]
Define
\[
v_{s_0, \ldots, s_n} = G_{s_0} \circ \cdots \circ G_{s_{n-1}}(v_{s_n})
\]
for $s_0, \ldots, s_n \in \Z$, $n \ge 1$. By Lemma~\ref{lem:alpha} and \eqref{eq:V},
\begin{equation}\label{eq:vinU}
v_s \in \Gamma_s \subset V, \qquad v_{s_0, \ldots, s_n} \in \Gamma_{s_0, \ldots, s_n} \subset V_{s_0, \ldots, s_{n-1}} \subset V.
\end{equation}
See Figure~\ref{fig:T}.

\begin{figure}[ht!]
\centerline{\includegraphics[width=0.6\textwidth]{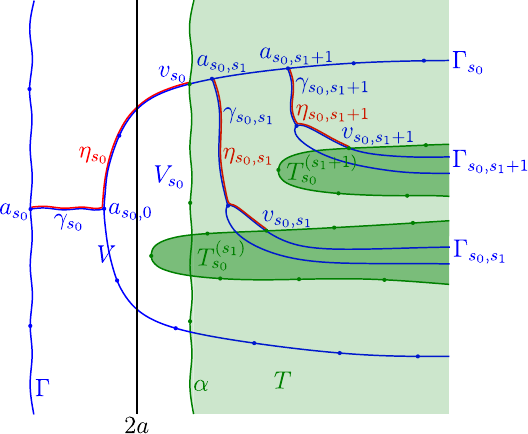}}
\caption{The sets $T_{s_0}^{(s_1)}$ and curves $\eta_{s_0, s_1}$.}
\label{fig:T}
\end{figure}

A crucial property is that for given $s_0 \in \Z$, subsequent points of the two-sided sequence $(v_{s_0, s_1})_{s_1 \in \Z}$ lie at a  uniformly bounded distance from each other, as described in the following lemma.

\begin{lemma}\label{lem:diameta} There exist constants $c_1,c_2>0$ such that for every $s_0, s_1 \in \Z$, 
\[
|v_{s_0, s_1} - v_{s_0, s_1+1}| < c_1 |G_{s_0}'(v_{s_1})| < c_2.
\]
\end{lemma}
\begin{proof}Let $t_0 > 0$ be such that
\[
\Gamma_s(t_0) = v_s, \qquad s \in \Z.
\]
Define
\[
\eta_{s_1} = \gamma_{s_1} \cup \Gamma_{s_1}|_{[0,t_0]}, \qquad \eta_{s_0, s_1} =  G_{s_0}(\eta_{s_1})
\]
for $s_0, s_1\in \Z$. Note that $\eta_{s_1}$ connects $a_{s_1}$ to $v_{s_1}$, while $\eta_{s_0, s_1}$ connects $a_{s_0, s_1}$ to $v_{s_0, s_1}$. By Lemma~\ref{lem:GammainL},
\[
\eta_{s_1} \setminus \{a_{s_1}\} \subset V, \qquad
\eta_{s_0, s_1} \setminus \{a_{s_0, s_1}\} \subset V_{s_0}.
\]
See Figure~\ref{fig:T}. Let 
\[
L_{s_1} = \eta_{s_1} \cup \Gamma([s_1,s_1+1]) \cup \eta_{s_1+1}, \qquad s_1 \in \Z,
\]
and note that $L_{s_1}$ is a curve connecting $v_{s_1}$ to $v_{s_1+1}$, such that 
\[
L_{s_1}\subset V \cup \Gamma, \qquad  L_{s_1} = L_0 + 2\pi i s_1.
\]
Since $\overline{\Sing(f)\cap U} \subset D$, there exists $\delta>0$ such that the maps $G_{s_0}$, $s_0 \in \Z$, can be extended univalently to $\{z \in \C: \dist(z, \Gamma) < \delta\}$. Hence, by Koebe's distortion theorem, the distortion of $G_{s_0}$ on $L_{s_1}$ is bounded by a constant independent of $s_0, s_1$. This provides the left-hand side inequality in the assertion of the lemma. Moreover, by \eqref{eq:Gamma<a} and Lemma~\ref{lem:alpha}, we have $\D(v_{s_1},a) \subset V$, so by Koebe's one-quarter theorem,
\[
\D\Big(v_{s_0,s_1}, |G_{s_0}'(v_{s_1})|\frac{a}{4}\Big) \subset G_{s_0}(\D(v_{s_1},a)) \subset V_{s_0}.
\]
However, $V_{s_0}$ does not contain vertical segments of length $2\pi$, as it is a logarithmic lift of $\mathcal V$. Hence, $|G_{s_0}'(v_{s_1})|a/4 \le 2\pi$, which provides the right-hand side inequality. 
\end{proof}

Let $T$ be the component of $\C \setminus \alpha$ with unbounded positive real part. By definition, $T$ is a simply connected unbounded domain such that $\bd T = \alpha$ is homeomorphic to the real line and
\begin{equation}\label{eq:TinH}
\overline{T} \subset \mathbb H_{>2a} \qquad \text{and} \qquad v_s \in \bd T, \quad s \in \Z.
\end{equation}
Hence, every component of $F^{-1}(\overline{T})$ is a simply connected set with the boundary homeomorphic to the real line, mapped by $F$ homeomorphically onto $\overline{T}$. 

We are going to compare the behaviour of inverse branches of the lifts $\Phi$ and $F$ near the points $v_{s_0, \ldots, s_n}$. First, we show that if $s_1$ is sufficiently large, then these points are contained in $T$. Recall that we write
\[
\Z_{\ge \sigma} = \{s \in \Z: s \ge \sigma\}
\]
for $\sigma \in \Z$.

\begin{lemma}[{\bf\boldmath Existence of the number $\sigma$}{}]\label{lem:hats} There exists $\sigma \in \Z$ such that 
\[
\gamma_{s_0, s_1} \cup \Gamma_{s_0, s_1}\cup V_{s_0, s_1} \subset T \qquad \text{for } \;s_0  \in \Z, \;s_1 \in \Z_{\ge \sigma}. 
\]
\end{lemma}
\begin{proof} As $\gamma_{s_0, s_1} \cup \Gamma_{s_0, s_1}\cup V_{s_0, s_1} = 
\gamma_{0, s_1} \cup \Gamma_{0, s_1}\cup V_{0, s_1} + 2\pi i s_0$ and $T+ 2\pi i s_0 = T$, we can assume that $s_0$ is fixed. By the properties listed in Subsection~\ref{subsec:U-E}, we have $\Gamma_{s_0, s_1}\cup V_{s_0, s_1} \subset \overline{\tilde V_{s_0, s_1}}$, where $\tilde V_{s_0, s_1}$ is the component of $\C \setminus \Gamma_{s_0, s_1}$ disjoint from $\gamma_{s_0, s_1}$. Hence, noting that $\bd \tilde V_{s_0, s_1} = \Gamma_{s_0, s_1}$ and $T \supset \mathbb H_{>x}$ for large $x >0$, one can see that to prove the lemma, it is sufficient to show $\inf\{\Re(w): w \in \gamma_{s_0, s_1} \cup \Gamma_{s_0, s_1}\} \to +\infty$ for $s_1 \to +\infty$, which is equivalent to $\inf\{|z|: z \in \exp(\gamma_{s_0, s_1}) \cup \exp(\Gamma_{s_0, s_1})\} \to +\infty$ for $s_1 \to +\infty$. 

Suppose that the assertion does not hold. Then there exist sequences $s^{(n)} \in \Z$ and $z_n \in \exp(\gamma_{s_0, s^{(n)}}) \cup \exp(\Gamma_{s_0, s^{(n)}})$, $n \ge 1$, such that the sequence $(s^{(n)})_{n=1}^\infty$ is strictly increasing, $s^{(n)} \to +\infty$ and $z_n \to z$ as $n \to \infty$, for some point $z \in \C$. Note that $\exp(\Gamma_{s_0, s^{(n)}})$ is a component of $f^{-1}(\omega)\cap U$, while $\exp(\gamma_{s_0, s^{(n)}})$ is a component of $f^{-1}(\kappa)\cap U$, where $\kappa = \exp(\gamma_{s_0}) \subset U \setminus D$ is a curve connecting $\bd D$ to $\omega$. Moreover, since $\gamma_{s_0, s^{(n)}} \cup \Gamma_{s_0, s^{(n)}} \subset V_{s_0} \cup \Gamma_{s_0}$ by Lemma~\ref{lem:GammainL}, the sets $\exp(\gamma_{s_0, s^{(n)}}) \cup \exp(\Gamma_{s_0, s^{(n)}})$ are pairwise disjoint for $n \ge 1$. 

We have $z_n \in f^{-1}(\kappa \cup \omega) \cap U$, so by continuity, $z \in f^{-1}(\kappa \cup \omega) \cap U$. Noting that $\kappa \cup \omega$ is disjoint from $\overline{\Sing(f|_U)}$, we see that a neighbourhood of $z$ is a univalent image of a neighbourhood of $f(z)$ under a branch of $f^{-1}$, hence it cannot intersect infinitely many different components of $f^{-1}(\kappa \cup \omega) \cap U$. This makes a contradiction, proving the lemma.
\end{proof}

Lemma~\ref{lem:hats} together with \eqref{eq:vinU} implies the following.

\begin{corollary}\label{cor:vinH}
$v_{s_0, \ldots, s_n} \in T$ for $s_0, \ldots, s_n \in \Z$,  $n \ge 1$, $s_1 \ge \sigma$. 
\end{corollary}

The following lemma gathers the main results of this section. We define some special inverse branches $H_{s_0}^{(s_1)}$ of $F$ and show how they act on the points $v_{s_0, \ldots, s_n}$. These branches will be used in the construction of a suitable expanding conformal repeller for an iterate of $F$, presented in Sections~\ref{sec:prelim_est}--\ref{sec:est_der}. 

\begin{lemma}[{\bf\boldmath Existence of the numbers $k(s_1)$ and maps $H_{s_0}^{(s_1)}$}{}]\label{lem:k(s)} For $s_0,s_1 \in \Z$ there exist a lifted tract $W_{s_0}^{(s_1)}$ of $f$ over infinity and a number $k(s_1) \in \Z$, such that for 
\[                                                                                                                                                          H_{s_0}^{(s_1)} = (F|_{W_{s_0}^{(s_1)}})^{-1} \colon \mathbb H_{>a} \to W_{s_0}^{(s_1)}, \qquad T_{s_0}^{(s_1)} = H_{s_0}^{(s_1)}(T),                                          
\]
the following hold.
\begin{enumerate}[$($a$)$]
\item $v_{s_0, s_1} \in \bd T_{s_0}^{(s_1)}$.
\item $H_{s_0}^{(s_1)} = H_0^{(s_1)} + 2\pi i s_0$.
\item $H_{s_0}^{(s_1)}(v_{k(s_1)}) = v_{s_0,s_1}$ and $(H_{s_0}^{(s_1)})'(v_{k(s_1)}) = G'_{s_0}(v_{s_1})$.
\item $H_{s_0}^{(s_1)}(v_{k(s_1), s_2,  \ldots, s_n}) = v_{s_0,s_1, \ldots, s_n}$
for $s_2, \ldots, s_n \in \Z$, $n \ge 2$, $s_2 \ge \sigma$. 
\end{enumerate}
\end{lemma}
\begin{proof}
By Lemma~\ref{lem:alpha}, we have $v_{s_1} \in \bd T \cap \Gamma_{s_1}$ for $s_1 \in \Z$. Consequently, by \eqref{eq:TinH},  for $s_0 \in \Z$ there holds $|f(e^{v_{s_0,s_1}})| = |e^{\Phi(v_{s_0,s_1})}| = |e^{v_{s_1}}| = e^{\Re(v_{s_1})} > e^a$, so $v_{s_0,s_1} \in F^{-1}(\mathbb H_{>a})$. This together with \eqref{eq:Phi-F} and \eqref{eq:vinU} implies $e^{v_{s_1}} = e^{\Phi(v_{s_0,s_1})} = e^{F(v_{s_0,s_1})}$. Therefore, $F(v_{s_0,s_1}) \in \exp^{-1}(e^{v_{s_1}}) = \{v_s\}_{s\in \Z}$. Moreover, $F(v_{s_0,s_1})$ does not depend on $s_0$, as $F$ is periodic with period $2\pi i$. It follows that for every $s_1 \in \Z$ there exists $k(s_1) \in \Z$, such that 
\[
F(v_{s_0,s_1}) = v_{k(s_1)}
\]
for every $s_0 \in \Z$. Consequently, there exist a lifted tract $W_{s_0}^{(s_1)}$ of $f$ over infinity and an inverse branch of $F$
\[
H_{s_0}^{(s_1)}\colon \mathbb H_{>a} \to W_{s_0}^{(s_1)},
\]
such that 
\[
H_{s_0}^{(s_1)}(v_{k(s_1)}) = v_{s_0,s_1}.
\]
We have $H_{s_0}^{(s_1)} = H_0^{(s_1)} + 2\pi i s_0$, since $F$ is periodic with period $2\pi i$. Obviously, $v_{s_0, s_1} \in \bd T_{s_0}^{(s_1)}$ for $T_{s_0}^{(s_1)} = H_{s_0}^{(s_1)}(T)$. See Figure~\ref{fig:T}.

Let 
\[
\tilde G_{s_0}(z) = G_{s_0}(z + 2\pi i (s_1- k(s_1)))
\]
for $z \in V$. Note that
\[
\tilde G_{s_0}(v_{k(s_1)}) = v_{s_0,s_1},
\]
which implies
\[
\Phi(\tilde G(v_{k(s_1)})) - F(\tilde G(v_{k(s_1)})) = v_{s_1} - v_{k(s_1)} = 2\pi i (s_1- k(s_1)).
\]
Moreover, by \eqref{eq:Phi-F}, \eqref{eq:vinU} and \eqref{eq:TinH}, we have $e^{\Phi(\tilde G(z))} = e^{F(\tilde G(z))}$ for $z$ in a small neighbourhood of $v_{k(s_1)}$, so $\Phi(\tilde G(z)) - F(\tilde G(z)) \in 2\pi i \Z$. Hence, by continuity, 
\[
z  + 2\pi i (s_1- k(s_1)) - F(\tilde G(z)) = \Phi(\tilde G(z)) - F(\tilde G(z)) = 2\pi i (s_1- k(s_1)),
\]
for $z$ in a neighbourhood of $v_{k(s_1)}$, which provides $z = F(\tilde G(z)) $ and, consequently, $H_{s_0}^{(s_1)}(z) =  \tilde G_{s_0}(z)$ for $z$ in a neighbourhood of $v_{k(s_1)}$. In particular, 
\[
(H_{s_0}^{(s_1)})'(v_{k(s_1)}) = \tilde G'_{s_0}(v_{k(s_1)}) = G'_{s_0}(v_{s_1}).
\]
By holomorphicity, $H_{s_0}^{(s_1)}(z) =  \tilde G_{s_0}(z)$ for $z$ in the connected component of $V \cap \mathbb H_{>a}$ containing $v_{k(s_1)}$. Let $t_0 > 0$ be such that $\Gamma_s(t_0) = v_s$, $s \in \Z$, and define
\[
Y = \Gamma_{k(s_1)}([t_0, +\infty)) \cup \bigcup_{n=2}^\infty \; \bigcup_{s_2, \ldots, s_n \in \Z, \, s_2 \ge \sigma} \gamma_{k(s_1), s_2, \ldots, s_n} \cup \Gamma_{k(s_1), s_2, \ldots, s_n}.
\]
By \eqref{eq:V}, \eqref{eq:TinH} and Lemma~\ref{lem:alpha}, the set $\Gamma_{k(s_1)}([t_0, +\infty))$ is a connected subset of $V\cap \mathbb H_{>a}$. Moreover, by Lemma~\ref{lem:GammainL} (applied for $n= 1, 2$), Lemma~\ref{lem:hats} and \eqref{eq:TinH}, for every $s_2 \in \Z_{\ge \sigma}$, the set
\[
\gamma_{k(s_1), s_2} \cup \Gamma_{k(s_1), s_2} \cup \bigcup_{n=3}^\infty\bigcup_{s_3, \ldots, s_n \in \Z} \gamma_{k(s_1), s_2, \ldots, s_n} \cup \Gamma_{k(s_1), s_2, \ldots, s_n}
\]
is a connected subset of $V_{k(s_1)} \cap \mathbb H_{>a} \subset V \cap \mathbb H_{>a}$. Since $\Gamma_{k(s_1)}([t_0, +\infty))$ connects $v_{k(s_1)}$ to $\gamma_{k(s_1), s_2}$ for every $s_2\in \Z_{\ge \sigma}$, we conclude that $Y$ is contained in the connected component of $V \cap \mathbb H_{>a}$ containing $v_{k(s_1)}$. Hence, $H_{s_0}^{(s_1)}(z) =  \tilde G_{s_0}(z)$ for every $z \in Y$. In particular, as $v_{k(s_1), s_2,  \ldots, s_n} \in Y$ for $s_1, \ldots, s_n \in \Z$, $n \ge 2$, $s_2 \in \Z_{\ge \sigma}$, we obtain
\[
H_{s_0}^{(s_1)}(v_{k(s_1), s_2,  \ldots, s_n}) = G_{s_0}(v_{s_1, \ldots, s_n}) = v_{s_0, \ldots, s_n}.
\]
This completes the proof.
\end{proof}

We end this section by indicating some properties of the branches $H_{s_0}^{(s_1)}$ from Lemma~\ref{lem:k(s)}. First, note that since for every $x > 0$, the real part of $F$ is bounded from above on $\mathbb H_{\le x} \cap W_{s_0}^{(s_1)}$, there holds
\begin{equation}\label{eq:to_infty}
\text{if } \quad z_n \in \mathbb H_{>a},\; |z_n| \to \infty, \quad\text{ then }\quad \Re(H_{s_0}^{(s_1)}(z_n)) \to +\infty.
\end{equation}
Contracting properties of the branches $H_{s_0}^{(s_1)}$ are described in the following lemma (assuming $a$ is chosen sufficiently large).

\begin{lemma}\label{lem:contract}
For every $z_1, z_2 \in \mathbb H_{>2a}$,
\[
|H_{s_0}^{(s_1)}(z_1) - H_{s_0}^{(s_1)}(z_2)| \le \frac{4\pi |z_1 - z_2|}{\sqrt{(\Re(z_1)-a)(\Re(z_2)-a)}} \le \frac{|z_1 - z_2|}{2}.
\]
\end{lemma}
\begin{proof}
This is a direct consequence of a standard estimate of the hyperbolic metric (see e.g.~\cite{coding}), providing
\[
|z_1 - z_2| \le 4 \pi \asinh \frac{|F(z_1) - F(z_2)|}{2 \sqrt{(\Re(F(z_1))-a)(\Re(F(z_2)) - a)}}
\]
for $z_1, z_2 \in W_{s_0}^{(s_1)}$.
\end{proof}

\section{Construction of a conformal repeller -- preliminary estimates}\label{sec:prelim_est}

In this and the following section we construct a suitable expanding conformal repeller for some high iterate of the map $F$. To this end, we define an iterated function system of its appropriate inverse branches on some rectangle $K \subset \mathbb H_{>2a}$. 

\subsection{\boldmath Definition of the set $K$}\label{subsec:K}

For simplification, denote
\[
H_\sigma = H_{k(\sigma)}^{(\sigma)}, \qquad T_\sigma = T_{k(\sigma)}^{(\sigma)}
\]
for $\sigma$ from Lemma~\ref{lem:hats} and $H_{k(\sigma)}^{(\sigma)}$, $T_{k(\sigma)}^{(\sigma)}$ from Lemma~\ref{lem:k(s)}. Fix a large number 
\begin{equation}\label{eq:b}
b>2a+1
\end{equation}
(in the sequel, we may need to increase $b$ several times, if necessary). Choose a point
\[
u = r +iy \in T_\sigma,
\]
where 
\begin{equation}\label{eq:r}
r > \max(b+1, \inf\{\Re(z): z\in T_\sigma\}).
\end{equation}
By \eqref{eq:to_infty}, increasing $r$ we can assume 
\begin{equation}\label{eq:KinT}
\mathbb H_{> r-2}, F^{-1}(\mathbb H_{> r-2}), F^{-2}(\mathbb H_{> r-2}), F^{-3}(\mathbb H_{> r-2}) \subset \mathbb H_{>2a}.
\end{equation}

Let
\[
K = \{z \in \C: \Re(z) \in [r-1, r+1], \; \Im(z) \in [y-d, y+d]\},
\]
where $d>1$ is so large that
\begin{equation}\label{eq:throughK}
\{z \in T_\sigma: \Re(z) \in [r-1, r+1]\} \subset 
\{z \in \C: \Im(z) \in [y-d+1, y+d-1]\}.
\end{equation}
See Figure~\ref{fig:K}.

\begin{figure}[ht!]
\centerline{\includegraphics[width=0.5\textwidth]{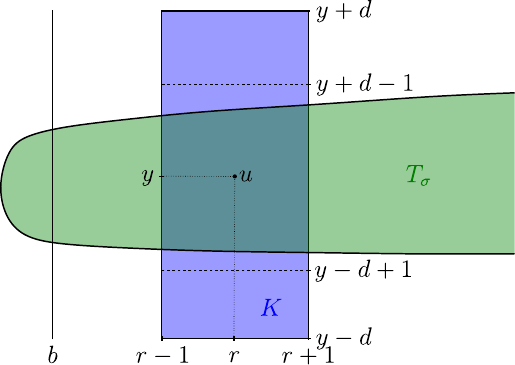}}
\caption{The rectangle $K$.}\label{fig:K}
\end{figure}

\subsection{\boldmath Definition and estimates for the branches $\psi_{j,l}$}
\label{subsec:psi}
To construct a suitable iterated function system on $K$, we first define some special inverse branches of $F^3$. More precisely, by \eqref{eq:KinT}, we can consider the inverse branches of the form
\[
\psi_{j, l}\colon\mathbb H_{> r-2} \to \C, \qquad \psi_{j, l} = H_{k(\sigma)}^{(j)} \circ H_{k(j)}^{(l)} \circ H_{k(l)}^{(\sigma)},\qquad  j, l \in \Z_{\ge \sigma}.
\]
In this subsection we describe the properties of the branches $\psi_{j, l}$.

\begin{lemma}\label{lem:distinct} The maps $\psi_{j, l}$, $j, l \in \Z_{\ge \sigma}$, are pairwise distinct.
\end{lemma}
\begin{proof} For simplification, denote
\begin{equation}\label{eq:u_s}
u_{l} = H_{k(l)}^{(\sigma)}(u), \qquad u_{j, l} = H_{k(j)}^{(l)}(u_{l})
\end{equation}
for $j, l \in \Z_{\ge \sigma}$ (see Figure~\ref{fig:U}). 

\begin{figure}[ht!]
\centerline{\includegraphics[width=0.99\textwidth]{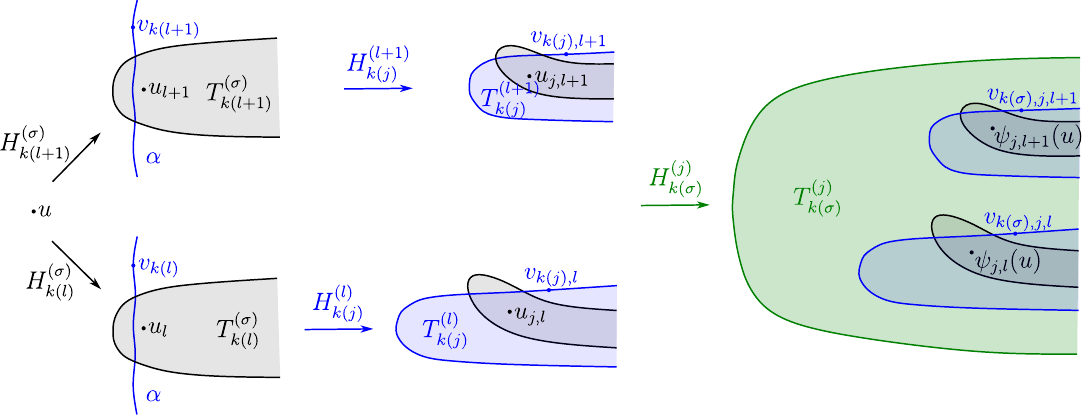}}
\caption{The location of the points $u_{l}$, $u_{j, l}$ and $\psi_{j, l}(u)$.}\label{fig:U}
\end{figure} 

Suppose $\psi_{j, l} = \psi_{j', l'}$ for some $j,l,  j', l' \in \Z_{\ge \sigma}$. Note that $\psi_{j, l}(u) \in T_{k(\sigma)}^{(j)}$, $\psi_{j', l'}(u) \in T_{k(\sigma)}^{(j')}$, where $T_{k(\sigma)}^{(j)}$, $T_{k(\sigma)}^{(j')}$ are either disjoint or equal. Therefore, $T_{k(\sigma)}^{(j)} = T_{k(\sigma)}^{(j')}$, which implies 
\begin{equation}\label{eq:H1}
H_{k(\sigma)}^{(j)} = H_{k(\sigma)}^{(j')}
\end{equation}
and, consequently, $u_{j, l} = u_{j', l'}$ by the univalency of the map. Analogously, $u_{j, l} \in T_{k(j)}^{(l)}$, $u_{j', l'} \in T_{k(j')}^{(l')}$, where $T_{k(j)}^{(l)}$, $T_{k(j')}^{(l')}$ are either disjoint or equal, which in turn implies $T_{k(j)}^{(l)} = T_{k(j')}^{(l')}$, 
\begin{equation}\label{eq:H2}
H_{k(j)}^{(l)} = H_{k(j')}^{(l')}
\end{equation}
and $u_{l} = u_{l'}$. Finally, $u_{l} \in T_{k(l)}^{(\sigma)}$, $u_{l'} \in T_{k(l')}^{(\sigma)}$, where $T_{k(l)}^{(\sigma)}$, $T_{k(l')}^{(\sigma)}$ are either disjoint or equal, which provides $T_{k(l)}^{(\sigma)} = T_{k(l')}^{(\sigma)}$. As $T_{k(l)}^{(\sigma)} = T_{k(l')}^{(\sigma)} + 2\pi i (k(l) - k(l'))$ by Lemma~\ref{lem:k(s)}, we have $k(l) = k(l')$. This together with \eqref{eq:H2} and Lemma~\ref{lem:k(s)} implies
\[
v_{k(j),l} = H_{k(j)}^{(l)}(v_{k(l)}) = H_{k(j')}^{(l')}(v_{k(l')}) = v_{k(j'),l'}.
\]
By the definition of $v_{k(j),l}$, $v_{k(j'),l'}$, we have $k(j) = k(j')$ and $l = l'$. Hence, by \eqref{eq:H1} and Lemma~\ref{lem:k(s)}, 
\[
v_{k(\sigma),j} = H_{k(\sigma)}^{(j)}(v_{k(j)}) = H_{k(\sigma)}^{(j')}(v_{k(j')}) = v_{k(\sigma),j'}
\]
and $j = j'$. This implies $(j, l) = (j', l')$, which ends the proof.
\end{proof}

The next lemma shows that the images of the point $u$ under the branches $\psi_{j,l}$ can be `approximated' by points $v_{k(\sigma), j, l}$, while subsequent points of the sequence $(v_{k(\sigma), j, l})_{l=\sigma}^\infty$ are located at a uniformly bounded distance to each other.

\begin{lemma}\label{lem:a-u} There exists $c>0$ such that for $j, l \in \Z_{\ge \sigma}$,
\[
|\psi_{j,l}(u) - v_{k(\sigma), j, l}|,  |v_{k(\sigma), j, l} - v_{k(\sigma), j, l+1}| < c
\]
and
\[
|\psi_{j,l}'(u)| \ge \frac 1 c | v_{k(\sigma), j, l} - v_{k(\sigma), j, l+1}|.
\]
\end{lemma}
\begin{proof} We continue using the notation from \eqref{eq:u_s} (cf.~Figure~\ref{fig:U}). 
By Lemma~\ref{lem:k(s)}, 
\begin{equation}\label{eq:dist-1}
|u_{l} - v_{k(l)}| = |H_{k(l)}^{(\sigma)}(u) - v_{k(l)}| = c_1
\end{equation}
for $c_1 = |H_\sigma(u) - v_{k(\sigma)}|$. Hence, by \eqref{eq:TinH}, \eqref{eq:KinT}, Corollary~\ref{cor:vinH} and Lemmas~\ref{lem:k(s)} and~\ref{lem:contract}, 
\begin{equation}\label{eq:dist-2}
\begin{aligned}
|u_{j,l} - v_{k(j), l}| &= 
|H_{k(j)}^{(l)} (u_{l}) - H_{k(j)}^{(l)} (v_{k(l)})| \le \frac{c_1}{2},\\
|\psi_{j,l}(u) - v_{k(\sigma), j, l}| &= 
|H_{k(\sigma)}^{(j)}(u_{j,l}) - H_{k(\sigma)}^{(j)}(v_{k(j), l})|\le\frac{c_1}{4}.
\end{aligned}
\end{equation}
Furthermore, by Lemma~\ref{lem:diameta}, 
\begin{equation}\label{eq:dist-2a}
|v_{k(j), l} - v_{k(j),l+1}| < c_2
\end{equation}
for some constant $c_2 > 0$. Hence, by Corollary~\ref{cor:vinH} and Lemmas~\ref{lem:k(s)} and~\ref{lem:contract},  
\begin{align*}
|v_{k(\sigma), j, l} - v_{k(\sigma), j, l+1}| &= 
|H_{k(\sigma)}^{(j)} (v_{k(j), l}) - H_{k(\sigma)}^{(j)}(v_{k(j), l+1})|\\
&\le\frac{1}{2}|v_{k(j), l} - v_{k(j), l+1}| < \frac{c_2}{2}.
\end{align*}
This shows the first part of the lemma. To prove the second one, note that by the chain rule,
\begin{equation}\label{eq:chain}
|\psi_{j,l}'(u)|= |(H_{k(l)}^{(\sigma)})'(u)| \, |(H_{k(j)}^{(l)})'(u_{l})|\, |(H_{k(\sigma)}^{(j)})'(u_{j, l})|.
\end{equation}
As $H_{k(l)}^{(\sigma)} = H_\sigma + 2\pi i (k(l) - k(\sigma))$ by Lemma~\ref{lem:k(s)}, we have 
\begin{equation}\label{eq:der-1}
|(H_{k(l)}^{(\sigma)})'(u)| = c_3
\end{equation}
for $c_3 = |H_\sigma'(u)|$. By \eqref{eq:TinH}, \eqref{eq:KinT} and \eqref{eq:dist-1}, Koebe's distortion theorem implies
\[
|(H_{k(j)}^{(l)})'(u_{l})| \ge c_4 |(H_{k(j)}^{(l)})'(v_{k(l)})|
\]
for some constant $c_4 > 0$, and by Lemmas~\ref{lem:diameta} and~\ref{lem:k(s)}, 
\[
|(H_{k(j)}^{(l)})'(v_{k(l)})| = |G'_{k(j)}(v_{l})| \ge c_5 |v_{k(j), l} - v_{k(j), l+1}|
\]
for some constant $c_5 > 0$. Hence,
\begin{equation}\label{eq:der-2}
|(H_{k(j)}^{(l)})'(u_{l})| \ge c_4 c_5 |v_{k(j), l} - v_{k(j), l+1}|. 
\end{equation}
Similarly, by \eqref{eq:KinT}, \eqref{eq:dist-2}, \eqref{eq:dist-2a}, Corollary~\ref{cor:vinH}, Lemma~\ref{lem:k(s)} and Koebe's distortion theorem,
\begin{align*}
|(H_{k(\sigma)}^{(j)})'(u_{j, l})| &\ge c_6 |(H_{k(\sigma)}^{(j)})'(v_{k(j), l})|\\ &\ge c_7 \frac{|H_{k(\sigma)}^{(j)}(v_{k(j), l}) - H_{k(\sigma)}^{(j)}(v_{k(j), l+1})|}{|v_{k(j), l} - v_{k(j), l+1}|}
= c_7 \frac{|v_{k(\sigma), j, l} - v_{k(\sigma), j, l+1}|}{|v_{k(j), l} - v_{k(j), l+1}|}
\end{align*}
for some constants $c_6, c_7 > 0$. This together with \eqref{eq:chain}, \eqref{eq:der-1} and \eqref{eq:der-2} proves the second part of the lemma.
\end{proof}

\section{Construction of a conformal repeller -- definition of the system}\label{sec:constr}

In this section we construct a suitable iterated function system on $K$, consisting of some inverse branches of $F^{m+3}$ of the form $\phi_{j, l} = H_\sigma^m \circ \psi_{j, l}$, for a sufficiently large $m$. 

The following lemma shows that arbitrarily large discs, provided they are contained in the part of the plane with sufficiently large real part, can be moved to the left half-plane $\mathbb H_{<b}$, by applying a high iterate of the branch $H_\sigma$.

\begin{lemma}\label{lem:Re<} For every $R > 0$ one can find $x = x(R) >0$, such that for every $z \in \mathbb H_{>x}$ there exists $m(z)$, which is the minimal number $m \in \N$ such that $H_\sigma^m$ is defined on $\D(z,R)$, with
\[
\D(z,R), H_\sigma(\D(z,R)), \ldots, H_\sigma^{m-1}(\D(z,R)) \subset \mathbb H_{>2a}, \qquad  H_\sigma^m(\D(z,R)) \subset \mathbb H_{< b}.
\]
Furthermore, for every $c>0$ there exists $q = q(c) \in \N$ such that for every $R > 0$, if $z_1, z_2 \in \mathbb H_{>x}$ for $x=x(R)$ and $|z_1 - z_2|< c$, then $|m(z_1) - m(z_2)| < q$. 
\end{lemma}
\begin{proof} Set 
\[
z_0 = 3a, \qquad R_0 =  4(|H_\sigma(z_0) - z_0|).
\]
Increasing the constant $b$ from \eqref{eq:b} if necessary, we can assume
\begin{equation}\label{eq:w-z0}
H_\sigma(\D(z_0, R_0) \cap \mathbb H_{>2a}) \subset \mathbb H_{< b-1}.
\end{equation}
By Lemma~\ref{lem:contract}, for every $w \in \mathbb H_{>2a}$,
\[
|H_\sigma(w) - z_0| \le |H_\sigma(w) - H_\sigma(z_0)| + |H_\sigma(z_0) - z_0| \le \frac{|w-z_0|}{2} + \frac{R_0}{4}.
\]
Hence, 
\begin{equation}\label{eq:/2}
\text{if } \quad w \in \mathbb H_{>2a} \setminus \D(z_0,R_0), \quad \text{then } \quad |H_\sigma(w) - z_0| \le  \frac{3}{4}|w-z_0|.
\end{equation}

Let $R >0$. Fix $m_0 \in \N$ such that 
\[
\frac{R}{2^{m_0}} < 1.
\]
By \eqref{eq:to_infty}, there exists $x>0$ such that $H_\sigma^{m_0}$ is defined on $\mathbb H_{>x-R}$ and
\[
\mathbb H_{>x-R}, H_\sigma(\mathbb H_{>x-R}), \ldots, H_\sigma^{m_0}(\mathbb H_{>x-R}) \subset \mathbb H_{>b}.
\]
Hence, for every $z \in \mathbb H_{>x}$, 
\[
\D(z,R), H_\sigma(\D(z,R)), \ldots, H_\sigma^{m_0}(\D(z,R)) \subset \mathbb H_{>b}\subset \mathbb H_{>2a}.
\]
Suppose there exists (the minimal) $m > m_0$ such that 
\[
\D(z,R), H_\sigma(\D(z,R)), \ldots, H_\sigma^{m-1}(\D(z,R)) \subset \mathbb H_{>2a}, \qquad H_\sigma^m(\D(z,R)) \not \subset \mathbb H_{>2a}
\]
(note that the first condition implies that $H_\sigma^m$ is defined on $\D(z,R)$). Then by Lemma~\ref{lem:contract},
\begin{equation}\label{eq:diamHm}
\diam H_\sigma^m(\D(z,R)) \le \frac{2R}{2^m} \le \frac{R}{2^{m_0}}  < 1.
\end{equation}
Take $v \in H_\sigma^m(\D(z,R)) \setminus \mathbb H_{>2a}$. By \eqref{eq:diamHm},
\[
H_\sigma^m(\D(z,R)) \subset \D(v, 1) \subset\mathbb H_{<2a+1}\subset\mathbb H_{<b},
\]
which shows that $m(z)$ exists and is equal to $m$. 

Suppose now
\[
H_\sigma^m(\D(z,R)) \subset \mathbb H_{>2a} \qquad \text{for every } m \ge 0.
\]
In particular, \eqref{eq:diamHm} holds and $H_\sigma^{m-1}(z) \in \mathbb H_{>2a}$ for all $m > m_0$. Note that if $H_\sigma^{m-1}(z)\notin \D(z_0,R_0)$ for every $m >m_0$, then using \eqref{eq:/2} inductively provides $H_\sigma^{m-1}(z) \to z_0$ as $m \to \infty$, which is a contradiction. Hence, $H_\sigma^{m-1}(z) \in \D(z_0, R_0)$ for some $m > m_0$. Then $\Re(H_\sigma^m(z)) < b-1$ by \eqref{eq:w-z0}, which together with \eqref{eq:diamHm} implies 
\[
H_\sigma^m(\D(z,R)) \subset \D(H_\sigma^m(z),1) \subset \mathbb H_{<b}.
\]
Then $m(z)$ can be defined as the minimal integer $m > m_0$ satisfying $H_\sigma^m(\D(z,R))\subset \mathbb H_{<b}$. This ends the proof of the first assertion of the lemma. 

To show the second one, consider $c>0$ and note that since the set $T_\sigma \cap \mathbb H_{\le b+1+c}$ is bounded, we have
\begin{equation}\label{eq:Tcap<b+1}
T_\sigma \cap \mathbb H_{<b+c} \subset \D(z_0, R(c))
\end{equation}
for some $R(c)>0$. Fix an integer $q  = q(c)> 1$ satisfying
\begin{equation}\label{eq:3/4^q<}
\bigg(\frac{3}{4}\bigg)^{q-2} R(c) < R_0.
\end{equation}

Take $R > 0$ and $z_1, z_2 \in \mathbb H_{>x}$ for $x = x(R)$, such that $|z_1 - z_2|< c$. By symmetry, we can assume $m(z_1) \ge m(z_2)$, so it is enough to prove $m(z_1) - m(z_2) < q$. Obviously, we can also suppose $m(z_1) - m(z_2) >2$. By Lemma~\ref{lem:contract} and the definition of $m(z_2)$, 
\[
\Re(H_\sigma^{m(z_2)}(z_1)) \le \Re(H_\sigma^{m(z_2)}(z_2)) + |H_\sigma^{m(z_2)}(z_1) - H_\sigma^{m(z_2)}(z_2)| \le b + \frac{|z_1-z_2|}{2^{m(z_2)}} < b+c, 
\]
so by \eqref{eq:Tcap<b+1},
\[
|H_\sigma^{m(z_2)}(z_1) - z_0|< R(c).
\]
On the other hand, for $m \in [m(z_2), m(z_1) -2]$ we have $H_\sigma^{m}(z_1) \in \mathbb H_{>2a}$ by the definition of $m(z_1)$. Moreover, $H_\sigma^{m}(z_1) \notin \mathbb D(z_0,R_0)$, because otherwise, by \eqref{eq:w-z0} and \eqref{eq:diamHm} we would have
\[
H_\sigma^{m+1}(\D(z_1,R)) \subset \D(H_\sigma^{m+1}(z_1),1) \subset \mathbb H_{<b},
\]
where $m + 1 < m(z_1)$, which contradicts the minimality of $m(z_1)$. Therefore, \eqref{eq:/2} implies
\[
R_0 \le |H_\sigma^{m(z_1)-2}(z_1) - z_0| \le \bigg(\frac{3}{4}\bigg)^{m(z_1) - m(z_2) - 2} R(c),
\]
which together with \eqref{eq:3/4^q<} provides $m(z_1) - m(z_2) < q$, finishing the proof.
\end{proof}

By Lemma~\ref{lem:a-u}, we can choose $R > 0$ so large that
\begin{equation}\label{eq:R>}
R > 2\bigg(\sup_{j, l \in \Z_{\ge \sigma}}\{|\psi_{j,l}(u) - v_{k(\sigma), j, l}| + |v_{k(\sigma), j, l} - v_{k(\sigma), j, l+1}|\} + \diam K \bigg)
\end{equation}
for the rectangle $K$ from Subsection~\ref{subsec:K}. Fix also a small number
\begin{equation}\label{eq:epsilon}
\varepsilon \in \Big(0,\frac 1 4\Big)
\end{equation}
and consider the number $b$ from \eqref{eq:b}.

The following lemma is a basic tool for the construction, showing that one can find an arbitrarily large set $\mathcal J$ of indices $j \in \Z_{\ge \sigma}$, `initial' indices $l = \ell_j\in \Z_{\ge \sigma}$ for $j \in \mathcal J$, and a number $m \in \N$, such that the sequences 
\[
(H_\sigma^m (\D(\psi_{j,l}(u),R)))_{l = \ell_j}^\infty, \qquad j \in \mathcal J
\]
consist of $\varepsilon$-small subsets of $T_\sigma$, where the `initial' ones
\[
H_\sigma^m (\D(\psi_{j,\ell_j}(u),R))
\]
are located in $\mathbb H_{<b}$, which is to the left of the rectangle $K$. 

\begin{lemma}\label{lem:seq}
For every $M \in \N$ one can find a set $\mathcal J \subset \Z_{\ge \sigma}$ of cardinality $M$, numbers $\ell_j \in \Z_{\ge \sigma}$, $j \in \mathcal J$, and a positive integer $m$, such that for every $j \in \mathcal J$ and every $l \in \Z_{\ge \ell_{j}}$,
\begin{enumerate}[$($a$)$]
\item $H_\sigma^m$ is defined on $\D(\psi_{j, l}(u),R)$,
\item $\D(\psi_{j, l}(u),R), H_\sigma(\D(\psi_{j, l}(u),R)), \ldots, H_\sigma^{m-1}(\D(\psi_{j, l}(u),R)) \subset \mathbb H_{>2a}$,
\item $H_\sigma^m(\D(\psi_{j, \ell_{j}}(u),R)) \subset \mathbb H_{<b}$,
\item $\diam H_\sigma^m(\D(\psi_{j, l}(u),R)) < \varepsilon$,
\end{enumerate}
for $R$ from \eqref{eq:R>} and $\varepsilon$ from \eqref{eq:epsilon}.
\end{lemma}
\begin{proof} Fix $M \in \N$. Let $x>0$ be the number fulfilling the assertion of Lemma~\ref{lem:Re<} for $R$ from \eqref{eq:R>}, and for $z \in \mathbb H_{>x}$ consider the number $m(z)$ from Lemma~\ref{lem:Re<}. Note that by \eqref{eq:to_infty}, increasing $x$ if necessary, we can assume $m(z) > m_0$, where
\begin{equation}\label{eq:R<epsilon}
\frac{2R}{2^{m_0}} < \varepsilon
\end{equation}
for $\varepsilon$ from \eqref{eq:epsilon}.

For every $j, l \in \Z_{\ge \sigma}$, we have $\Re(u_{j, l})  = \Re(u_{\sigma, l})\to +\infty$ as $l \to +\infty$. Hence, by \eqref{eq:to_infty}, there exists $\sigma' \in \Z_{\ge \sigma}$, such that $\psi_{j, l}(u) \in \mathbb H_{>x}$ for every $j \in \Z_{\ge \sigma}$, $l \in \Z_{\ge \sigma'}$. Denote
\[
m_{j,l} = m(\psi_{j, l}(u)), \qquad j \in \Z_{\ge \sigma},\; l \in \Z_{\ge \sigma'}.
\]
By Lemma~\ref{lem:Re<} and \eqref{eq:R<epsilon}, the map $H_\sigma^{m_{j,l}}$ is defined on $\D(\psi_{j, l}(u),R)$, with
\begin{align}\label{eq:(a)}
&\D(\psi_{j, l}(u),R), H_\sigma(\D(\psi_{j, l}(u),R)), \ldots, H_\sigma^{m_{j,l}-1}(\D(\psi_{j, l}(u),R)) \subset \mathbb H_{>2a},\\
\label{eq:(b)}
&H_\sigma^{m_{j,l}}(\D(\psi_{j, l}(u),R)) \subset \mathbb H_{<b},\\
\label{eq:(c)}
&\diam H_\sigma^{m_{j,l}}(\D(\psi_{j, l}(u),R)) < \diam H_\sigma^{m_0}(\D(\psi_{j, l}(u),R)) < \varepsilon.                                                                                                                                                                  
\end{align}
Furthermore, Lemma~\ref{lem:a-u} implies that there exists a constant $c>0$ such that 
\[
|\psi_{j, l}(u) - \psi_{j, l+1}(u)| < c.
\]
Consider the number $q = q(c)$ from Lemma~\ref{lem:Re<} and let
\[
\tilde m = \max\{m_{j,\sigma'}: j = \sigma, \ldots, \sigma + qM - 1\}.
\]
By Lemma~\ref{lem:Re<}, for $j \in \{\sigma, \ldots, \sigma + qM - 1\}$, we have $|m_{j, l} - m_{j,l+1}| < q$ for every $l \in \Z_{\ge \sigma'}$. Moreover, as $\Re(u_{j, l}) \to +\infty$ as $l \to +\infty$, the assertion \eqref{eq:to_infty} implies $m_{j,l} \to +\infty$ as $l \to +\infty$. Consequently, one can find $\ell_j \in \Z_{\ge \sigma'}$ such that $m_{j, \ell_j} \in [\tilde m, \tilde m + q - 1]$ and $m_{j,l} \ge m_{j, \ell_j}$ for every $l \ge \ell_j$. As there are $qM$ numbers $m_{j, \ell_j}$, $j \in \{\sigma, \ldots, \sigma + qM - 1\}$, which can attain only $q$ different values, one can find 
a set $\mathcal J \subset \{\sigma, \ldots, \sigma + qM - 1\}$ of cardinality $M$ and an integer $m \in [\tilde m, \tilde m + q - 1]$, such that $m_{j, \ell_j} = m$ for every $j \in \mathcal J$. Consequently,
\[
m_{j, \ell_j} = m \qquad \text{and} \qquad m_{j,l} \ge m \qquad \text{for every } j \in \mathcal J, \; l \in \Z_{\ge \ell_{j}}.
\]
This together with \eqref{eq:(a)}, \eqref{eq:(b)} and \eqref{eq:(c)} provides the assertions of the lemma.
\end{proof}

Now we can define a suitable iterated function system on $K$. Take a large $M \in \N$ and consider the set $\mathcal J$, numbers $\ell_j$, $j \in \mathcal J$ and $m$ from Lemma~\ref{lem:seq} suited to $M$. By \eqref{eq:KinT}, \eqref{eq:R>} and Lemma~\ref{lem:seq}, the inverse branches of $F^{m+3}$, given by
\[
\phi_{j, l}: \mathbb H_{> r-2} \to \C, \qquad \phi_{j, l} = H_\sigma^m \circ \psi_{j, l}, \qquad j \in \mathcal J, \; l \in \Z_{\ge \ell_{j}}
\]
are well-defined. In particular, $\phi_{j, l}(K)$ are well-defined.

Let
\[
\mathcal K = \{(j, l): j \in \mathcal J, \; l \in \Z_{\ge \ell_{j}}, \; \phi_{j, l}(K) \subset K\}.
\]
Note that $\mathcal K$ is a finite set (otherwise, $F^{-(m+3)}(u)$ would have an accumulation point in $K$). Furthermore, Lemma~\ref{lem:distinct} implies the following. 

\begin{lemma}\label{lem:distinct1} The maps 
$\phi_{j, l}$, $(j, l) \in \mathcal K$, are pairwise distinct inverse branches of $F^{m+3}$. In particular, the sets $\phi_{j, l}(K)$, $(j, l) \in \mathcal K$, are pairwise disjoint.
\end{lemma}

Moreover, by Lemmas~\ref{lem:contract} and~\ref{lem:seq} together with \eqref{eq:R>},
\[
|\phi_{j, l}'| < \frac 1 {2^{m+3}} < 1 \qquad \text{for }\; (j, l) \in \mathcal K.
\]
Hence, $\{\phi_{j, l}|_K\}_{(j, l) \in \mathcal K}$ is an iterated function system of conformal contractions on $K$, with the attractor
\[
\Lambda = \bigcap_{n=1}^\infty \bigcup_{(j_1, l_1), \ldots, (j_n, l_n) \in \mathcal K} \phi_{j_n, l_n} \circ \cdots \circ \phi_{j_1, l_1} (K).
\]
In particular, $\Lambda$ is a compact set, such that $F^{m+3}(\Lambda) = \Lambda$. Hence, $\Lambda$ is an expanding conformal repeller for the map $F^{m+3}$.

\section{Estimate of pressure and proof of Theorem~A}\label{sec:est_der}

In this section we show that the Hausdorff dimension of the repeller $\Lambda$ is larger than~$1$, using the tools of thermodynamic formalism. To this end, one should estimate moduli of the derivatives of the branches $\phi_{j, l}$ on $K$. This is done in the two following lemmas. Set
\[
w_{j, l} = H_\sigma^m(v_{k(\sigma), j, l}), \qquad j \in \mathcal J, \; l \in \Z_{\ge \ell_{j}}
\]
and note that by \eqref{eq:R>} and Lemmas~\ref{lem:k(s)} and~\ref{lem:seq}, the points $w_{j, l}$ are well-defined and
\[
w_{j, l} = v_{\scriptsize k(\sigma), \underbrace{\sigma, \ldots, \sigma}_{m \text{ times}}, j, l}.
\]

\begin{lemma}\label{lem:der-phi} There exists $c > 0$, which does not depend on $M$, such that 
\[
|\phi_{j, l}'(u)| \ge c |w_{j, l} - w_{j, l+1}|, \qquad j \in \mathcal J, \; l \in \Z_{\ge \ell_{j}}.
\]
\end{lemma}
\begin{proof} 
By the chain rule and Lemma~\ref{lem:a-u}, 
\begin{equation}\label{eq:der-phi}
|\phi_{j, l}'(u)| = |\psi_{j, l}'(u)|\, |(H_\sigma^m)'(\psi_{j, l}(u))|\ge c_1 |v_{k(\sigma), j, l} - v_{k(\sigma), j, l+1}|\,|(H_\sigma^m)'(\psi_{j, l}(u))|
\end{equation}
for some constant $c_1 > 0$. By Lemma~\ref{lem:seq}, the map $H_\sigma^m$ is defined on $\D(\psi_{j, l}(u),R)$. Hence, by Koebe's distortion theorem, there exists $c_2 > 0$, such that 
\[
\frac{1}{c_2} \le \frac{|(H_\sigma^m)'(z_1)|}{|(H_\sigma^m)'(z_2)|} \le c_2 \qquad \text{for }\; z_1, z_2 \in \D\Big(\psi_{j, l}(u),\frac R 2\Big).
\]
Moreover, $v_{k(\sigma), j, l}, v_{k(\sigma), j, l+1} \in \D(\psi_{j, l}(u),R/2)$ by \eqref{eq:R>}. Consequently, 
\begin{align*}
|(H_\sigma^m)'(\psi_{j, l}(u))| &\ge \frac{1}{c_2} |(H_\sigma^m)'(v_{k(\sigma), j, l})|\\ &\ge \frac{1}{c_2^2} \frac{|H_\sigma^m(v_{k(\sigma), j, l})- H_\sigma^m(v_{k(\sigma), j, l+1})|}{|v_{k(\sigma), j, l}- v_{k(\sigma), j, l+1}|}\\ &= \frac{1}{c_2^2} \frac{|w_{j, l} - w_{j, l+1}|}{|v_{k(\sigma), j, l}- v_{k(\sigma), j, l+1}|}.
\end{align*}
This together with \eqref{eq:der-phi} implies
\[
|\phi_{j, l}'(u)| \ge \frac{c_1}{c_2^2} |w_{j, l} - w_{j, l+1}|.
\]
As $c_1, c_2$ do not depend on $M$, the proof is finished.
\end{proof}

\begin{lemma}\label{lem:sum-v}
For every $j \in \mathcal J$,
\[
\sum_{l: (j, l) \in \mathcal K} |w_{j, l} - w_{j, l+1}| > 1.
\]
\end{lemma}
\begin{proof} By \eqref{eq:R>}, for every $j \in \mathcal J$ we have $v_{k(\sigma), j, \ell_{j}} \in \D(\psi_{j,\ell_{j}}(u),R)$, so by Lemma~\ref{lem:seq} together with \eqref{eq:r}, 
\[
\Re(w_{j, \ell_{j}}) = \Re(H_\sigma^m(v_{k(\sigma), j, \ell_{j}})) < b < r-1.
\]
Since by \eqref{eq:to_infty}, the sequence $(\Re(w_{j, l}))_{l = \ell_{j}}^{\infty}$ tends to $+\infty$ as $l \to \infty$, there holds
\[
\Re(w_{j, l}) > r+1 \qquad \text{for sufficiently large } l \in \Z.
\]
Moreover, $v_{k(\sigma), j, l}, v_{k(\sigma), j, l+1} \in \D(\psi_{j,l}(u),R)$ by \eqref{eq:R>}, so Lemma~\ref{lem:seq} implies
\[
|w_{j, l} - w_{j, l+1}| < \varepsilon, \qquad l \in \Z_{\ge \ell_{j}}.
\]
Note also that 
\[
w_{j, l} \in T_\sigma.
\]
These facts together with \eqref{eq:throughK} and \eqref{eq:epsilon} imply that 
there exist $\ell^-_{j}, \ell^+_{j}\in \Z$ such that $\ell_{j} \le \ell^-_{j} \le \ell^+_{j}$ and 
\begin{align*}
\Re(w_{j, \ell^-_{j}}) &\in [r - 1 + \varepsilon, r - 1  + 2\varepsilon),\\
\Re(w_{j, l}) &\in [r - 1 + \varepsilon, r + 1  -\varepsilon] \qquad \text{for } \; l \in \{\ell^-_{j}, \ldots, \ell^+_{j}\},\\
\Re(w_{j, \ell^+_{j}}) &\in (r + 1 -2 \varepsilon, r + 1 - \varepsilon]
\end{align*}
(see Figure~\ref{fig:K-phi}).

\begin{figure}[ht!]
\centerline{\includegraphics[width=0.5\textwidth]{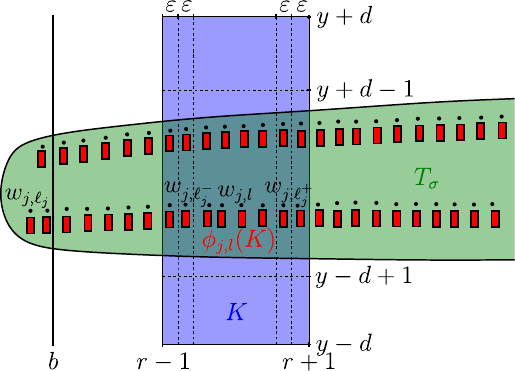}}
\caption{The points $w_{j, \ell_{j}}, w_{j, \ell^-_{j}}, w_{j, \ell^+_{j}}$.}\label{fig:K-phi}
\end{figure}
Furthermore, $\psi_{j,l}(K) \subset \D(\psi_{j,l}(u),R)$ by \eqref{eq:KinT}, \eqref{eq:R>} and Lemma~\ref{lem:contract}, so Lemma~\ref{lem:seq} implies $\phi_{j,l}(K) \subset \D(w_{j,l}, \varepsilon)$, which together with \eqref{eq:throughK} gives
\[
(j,l) \in \mathcal K \quad \text{for } l \in \{\ell^-_{j}, \ldots, \ell^+_{j}\}.
\]
Consequently, using \eqref{eq:epsilon}, we obtain
\[
\sum_{l: (j, l) \in \mathcal K} |w_{j, l} - w_{j, l+1}| \ge \sum_{l = \ell^-_{j}}^{\ell^+_{j}-1} |w_{j, l} - w_{j, l+1}| \ge |w_{j, \ell^-_{j}} - w_{j, \ell^+_{j}}|\\
\ge \Re(w_{j, \ell^+_{j}}) - \Re(w_{j, \ell^-_{j}}) > 2 - 4\varepsilon > 1.
\]
\end{proof}

Now we are ready to prove Theorem~A.

\begin{proof}[Proof of Theorem~\rm A]
By Lemma~\ref{lem:distinct1}, the topological pressure function for the map $F^{m+3}$ on the conformal repeller $\Lambda$ (see Subsection~\ref{subsec:bowen}) is equal to
\[
P(F^{m+3}|_{\Lambda}, t) = \lim_{n\to \infty} \frac 1 n \ln \sum_{(j_1, l_1), \ldots, (j_n, l_n) \in \mathcal K} |(\phi_{j_n, l_n} \circ \cdots \circ \phi_{j_1, l_1})'(w)|^t, \qquad t>0
\]
for $w \in \Lambda$. Obviously,
\[
P(F^{m+3}|_{\Lambda}, 1) \ge \ln \sum_{(j, l) \in \mathcal K} \inf_\Lambda|\phi_{j, l}'|.
\]
Since the maps $\phi_{j,l}$ are defined on $\mathbb H_{>r-2}$, while $K \subset \mathbb H_{\ge r-1}$, Koebe's distortion theorem implies
\[
\sum_{(j, l) \in \mathcal K} \inf_\Lambda|\phi_{j, l}'| > c_1 \sum_{(j, l) \in \mathcal K} |\phi_{j, l}'(u)|
\]
for some constant $c_1>0$. By Lemmas~\ref{lem:der-phi} and~\ref{lem:sum-v},
\[
\sum_{(j, l) \in \mathcal K} |\phi_{j, l}'(u)|  \ge c_2 \sum_{(j, l) \in \mathcal K} |w_{j, l} - w_{j, l+1}| \ge c_2\sum_{j \in \mathcal J}\; \sum_{l: (j, l) \in \mathcal K} |w_{j, l} - w_{j, l+1}| > c_2M,
\]
for some constant $c_2>0$. Noting that $M$ can be chosen to be arbitrarily large, while $c_1, c_2 > 0$ do not depend on $M$, we conclude that for sufficiently large $M$,  
\[
P(F^{m+3}|_{\Lambda}, 1) > \ln (c_1c_2M) > 0.
\]
Since $t \mapsto P(F^{m+3}|_{\Lambda}, t)$ is continuous, the unique zero of the pressure function is larger than~$1$. Hence, Bowen's formula implies that the Hausdorff dimension of $\Lambda$ is larger than~$1$. 

Let
\[
X = \bigcup_{n=0}^{m+2} \exp(F^n(\Lambda)) = \bigcup_{n=0}^{m+2}f^n(\exp(\Lambda)).
\]
It is immediate that $X$ is an expanding conformal repeller for the map $f$. Furthermore, $\dim_H X = \dim_H \Lambda > 1$, as the Hausdorff dimension is preserved by non-constant holomorphic maps.

To end the proof of Theorem~A, it is sufficient to show $X \subset \bd U$. Since $f(\bd U) \subset \bd U$, it is enough to prove
\begin{equation}\label{eq:LambdainU}
\exp(\Lambda) \subset \bd U.
\end{equation}

To show \eqref{eq:LambdainU}, take $z \in \Lambda$ and $n \ge 2$. Then 
\[
z = \phi_{j_n, l_n} \circ \cdots \circ \phi_{j_1, l_1}(z_n)
\]
for some $z_n \in K$ and $(j_1, l_1), \ldots, (j_n, l_n) \in \mathcal K$, 
so by Lemmas~\ref{lem:contract} and~\ref{lem:seq} together with \eqref{eq:KinT} and \eqref{eq:R>},
\begin{equation}\label{eq:z-phi<}
|z - \phi_{j_n, l_n} \circ \cdots \circ \phi_{j_1, l_1}(u)|= |\phi_{j_n, l_n} \circ \cdots \circ \phi_{j_1, l_1}(z_n) - \phi_{j_n, l_n} \circ \cdots \circ \phi_{j_1, l_1}(u)|< \frac{\diam K}{2^{(m+3)n}}.
\end{equation}
Let
\[
v^{(n)} = \phi_{j_n, l_n} \circ \cdots \circ \phi_{j_2, l_2}\circ H_\sigma^m (v_{k(\sigma), j_1, l_1}).
\]
By Lemma~\ref{lem:k(s)} and \eqref{eq:vinU},
\begin{equation}\label{eq:v^n}
v^{(n)} = v_{k(\sigma), \underline{\sigma}_m, j_n, l_n, \underline{\sigma}_{m+1}, j_{n-1}, l_{n-1}, \ldots, \underline{\sigma}_{m+1}, j_1, l_1} \in V \subset \exp^{-1}(U),
\end{equation}
where we write $\underline{\sigma}_q$ for $\underbrace{\sigma, \ldots, \sigma}_{q \text{ times}}$, $q = m, m+1$. 
Hence, by Lemmas~\ref{lem:contract}, \ref{lem:a-u} and~\ref{lem:seq} together with \eqref{eq:KinT} and \eqref{eq:R>},
\begin{equation}\label{eq:phi-v<}
\begin{aligned}
&|\phi_{j_n, l_n} \circ \cdots \circ \phi_{j_1, l_1}(u) - v^{(n)}|\\&= |\phi_{j_n, l_n} \circ \cdots \circ \phi_{j_2, l_2}\circ H_\sigma^m (\psi_{j_1, l_1}(u)) - \phi_{j_n, l_n} \circ \cdots \circ \phi_{j_2, l_2}\circ H_\sigma^m (v_{k(\sigma), j_1, l_1})|\\
&< \frac{c}{2^{(m+3)n-3}}
\end{aligned}
\end{equation}
for some constant $c > 0$. Using \eqref{eq:z-phi<} and \eqref{eq:phi-v<}, we obtain
\begin{equation}\label{eq:z-v<}
|z - v^{(n)}|< \frac{8c + \diam K}{2^{(m+3)n}},
\end{equation}
where the right hand side of the inequality tends to $0$ as $n \to \infty$, so $e^{v^{(n)}} \to e^z$. By \eqref{eq:v^n}, we have $e^{v^{(n)}} \in U$, so $e^z \in \overline{U}$. 
As $z \in \Lambda$, there holds $|(F^{(m+3)n})'(z)| \to \infty$ for $n \to \infty$, so $|(f^{(m+3)n})'(e^z)| \to \infty$ for $n \to \infty$, which shows $e^z \in J(f)$. Consequently, $e^z \in \bd U$. This proves \eqref{eq:LambdainU}, which implies that the hyperbolic dimension of $\bd U$ is larger than~$1$ and ends the proof of Theorem~A. 
\end{proof}

Finally, to prove Remark~\ref{rem:event_hyp}, it is enough to notice that $\Lambda \subset \mathbb H_{>a}$, where $a$ can be chosen to be arbitrarily large.

\section{Proof of Theorem~B}\label{sec:parab}

Let $U$ be an immediate component of the basin of a $p$-periodic parabolic point $\zeta$, satisfying the conditions of Theorem~B. In this section we indicate how to prove Theorem~B by adapting the arguments from Sections~\ref{sec:top struct}--\ref{sec:est_der} to the parabolic case. We only sketch the proof, highlighting the differences compared to the attracting case and leaving the details to the reader.

As previously, we can assume $p=1$. First, we prove a parabolic analogue of Lemma~\ref{lem:D}. 

\begin{lemma}\label{lem:D-parab} There exists a simply connected domain $D \subset U$ with a Jordan boundary, such that $\zeta \in \bd U$, $\overline{D} \setminus \{\zeta\} \subset U$, $\overline{\Sing(f|_U)} \subset D$ and $\overline{f(D)} \setminus \{\zeta\} \subset D$. Furthermore, $E = f^{-1}(D) \cap U$ is an unbounded simply connected domain containing $\overline D\setminus \{\zeta\}$, such that the components of $\bd E \cap (U \cup f^{-1}(\zeta))$ are homeomorphic to the real line, and are mapped by $f$ onto $\bd D$ as infinite degree coverings. Every component $\mathcal Y$ of $U \setminus \overline{E}$ is a simply connected domain, mapped univalently by $f$ onto $U \setminus \overline{D}$, such that $\bd\mathcal Y \cap U\subset\bd E$. 
\end{lemma}
\begin{proof} Let
\[
D_n = \{z \in U: \rho_U(z, f^{n}(z_0)) < r\}, \qquad n \ge 0,
\]
for a fixed $z_0 \in U$ and a large $r > 0$, where $\rho_U$ is the hyperbolic metric on $U$. Then $D_n$ is a simply connected domain with analytic Jordan boundary. As by assumption, $\overline{\Sing(f|_U)}$ is a compact subset of $U$, we have $\overline{\Sing(f|_U)} \subset D_0$ if $r$ was chosen sufficiently large. Increasing $r$, we can also assume $f(z_0) \in D_0$, which implies that $\bigcup_{n=0}^\infty D_n$ is connected. Moreover, $\overline{D_n} \subset D_{n+1}$ by the Schwarz--Pick lemma. 

Let $D$ be the interior of the union of the set $\overline{\bigcup_{n=0}^\infty D_n}$ and all components of its complement, which are contained in $U$.
Then $D$ is a simply connected domain with a Jordan boundary, such that $\zeta \in \bd U$, $\overline{D} \setminus \{\zeta\} \subset U$, $\overline{\Sing(f|_U)} \subset D$ and $\overline{f(D)} \setminus \{\zeta\} \subset D$. Repeating the arguments of the proof in the attracting case, we show that $E = f^{-1}(D) \cap U$ is an unbounded simply connected domain containing $\overline D\setminus \{\zeta\}$.
See Figure~\ref{fig:D-parab}.

\begin{figure}[ht!]
\centerline{\includegraphics[width=0.95\textwidth]{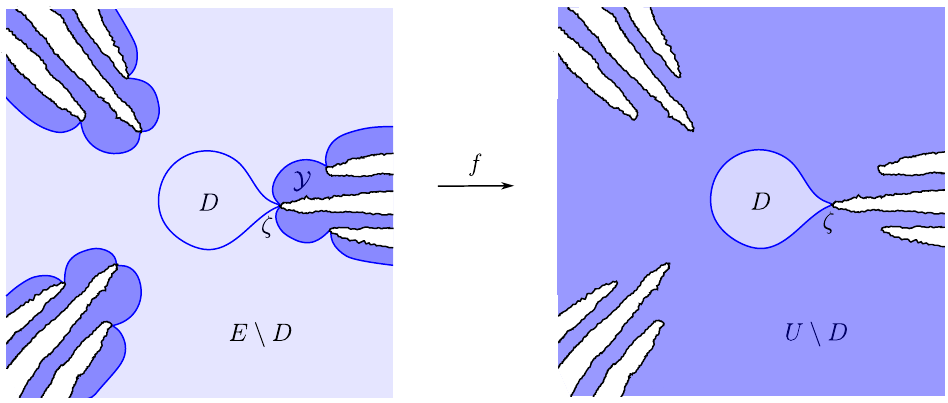}}
\caption{The sets $D$ and $E$ in the parabolic case.}\label{fig:D-parab}
\end{figure}

As a neighbourhood of $\bd D \setminus \{\zeta\}$ is disjoint from $\overline{\Sing(f|_U)}$, each component of $f^{-1}(\bd D \setminus \{\zeta\}) \cap U$ is a curve homeomorphic to $(0,1)$, mapped by $f$ homeomorphically onto $\bd D \setminus \{\zeta\}$, with each of its two `ends' converging to a preimage of $\zeta$ under $f$, or to infinity. Since we assume $\zeta \notin \overline{\Sing(f)}$, in fact each of the two `ends' converges to a different preimage of $\zeta$, and this preimage is not a critical point of $f$. This easily implies the remaining assertions of the lemma.
\end{proof}

As previously, we can assume $\zeta = 0$. Then by Lemma~\ref{lem:D-parab}, $\Gamma = \exp^{-1}(\bd D)$ is a curve homeomorphic to the real line, such that $\Gamma = \Gamma + 2\pi i $. The set 
\[
V = \exp^{-1}(U \setminus \overline D) = \bigcup_{s \in \Z}V^{(s)}
\]
has now infinitely many connected components $V^{(s)}$, $s \in \Z$, such that $V^{(s)} = V^{(0)} + 2 \pi i s$, and 
\[
\bd V^{(s)} \cap \bd V^{(s+1)} = \{a_s\}
\]
for a point $a_s \in \exp^{-1}(\zeta) \subset \Gamma$ with $a_s =  a_0 + 2\pi is$. See Figure~\ref{fig:G-parab}.

Lemma~\ref{lem:D-parab} implies that the map $f$ on $U \setminus E$
can be lifted by $\exp$ to a map 
\[
\Phi\colon \exp^{-1}(U \setminus E) \to \exp^{-1}(U \setminus D),
\]
such that $\Phi(z+2\pi i) = \Phi(z)$ and $\Phi(a_0) = a_0$. Fix $\omega$ to be a component of $\bd E$. By Lemma~\ref{lem:D-parab}, the components of $\exp^{-1}(U \setminus \overline{E})$, whose boundaries intersect $\exp^{-1}(\omega)$, can be denoted by $V^{(q)}_s$, $s,q \in \Z$, such that $V^{(q)}_s$ is a simply connected domain contained in $V$,
\[
V^{(q)}_s = V^{(q)}_0 + 2\pi i s
\]
and $V^{(q)}_s$ is mapped univalently by $\Phi$ onto $V^{(q)}$. Furthermore, the components of $\exp^{-1}(\omega)$ can be labeled as $\Gamma_s$, $s \in \Z$, where $\Gamma_s$ is a curve homeomorphic to the real line, with unbounded positive real part,
\[
\Gamma_s = \exp^{-1}(\omega) \cap \bd \bigg(\bigcup_{q\in \Z} V^{(q)}_s\bigg),
\]
$\Gamma_s = \Gamma_0 + 2\pi i s$, and $\Gamma_s$ is mapped homeomorphically by $\Phi$ onto $\Gamma$. A difference compared to the attracting case is that 
\[
\Gamma_s \cap \Gamma = \{a_s\},
\]
which means that the curves $\gamma_s$ from Section~\ref{sec:top struct} are now degenerated to the points $a_s$. See Figure~\ref{fig:G-parab}.
\begin{figure}[ht!]
\centerline{\includegraphics[width=0.65\textwidth]{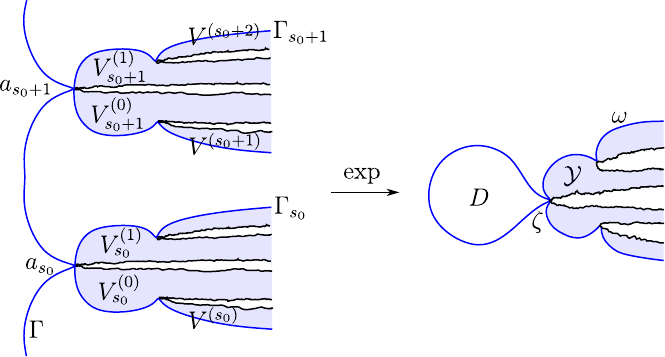}}
\caption{The domains $V^{(s)}$, $V^{(s)}_q$, points $a_s$ and curves $\Gamma$, $\Gamma_s$ in the parabolic case.}\label{fig:G-parab}
\end{figure}

Define
\[
G_s = (\Phi|_{\bigcup_{q\in \Z} V^{(q)}_s \cup \Gamma_s})^{-1} \colon V \cup \Gamma \to \bigcup_{q\in \Z} V^{(q)}_s \cup \Gamma_s, \qquad s,q \in \Z. 
\]
Then
\[
G_s = G_0 + 2\pi is
\]
and $G_s$ is a homeomorphism, univalent on each domain $V^{(q)}$. 

For $s_0, \ldots, s_n \in \Z$, $n \ge 1$, define
\[
V_{s_0, \ldots, s_n} = G_{s_0} \circ \cdots \circ G_{s_n}(V), \; \Gamma_{s_0, \ldots, s_n} = G_{s_0} \circ \cdots \circ G_{s_n}(\Gamma), \; a_{s_0, \ldots, s_n} = G_{s_0} \circ \cdots \circ G_{s_{n-1}}(a_{s_n}).
\]
The following analogue of Lemma~\ref{lem:GammainL} holds.

\begin{lemma}\label{lem:GammainL-parab} 
For every $s_0 \in \Z$,
\[
\Gamma_{s_0} \cup 
\bigcup_{k=1}^\infty\bigcup_{s_1, \ldots, s_k \in \Z} \Gamma_{s_0, \ldots, s_k}
\]
is a connected subset of $V \cup \bigcup_{k = 0}^\infty \{a_{s_0, \ldots, s_k}: s_0, \ldots, s_k \in \Z\}$. 

For every $s_0, \ldots, s_n \in \Z$, $n \ge 1$,
\[
\Gamma_{s_0, \ldots, s_n} \cup \bigcup_{k=1}^\infty\bigcup_{s_{n+1}, \ldots, s_{n+k} \in \Z} \Gamma_{s_0, \ldots, s_{n+k}}
\]
is a connected subset of $V_{s_0,\ldots, s_{n-1}} \cup \bigcup_{k = 0}^\infty \{a_{s_0, \ldots, s_{n+k}}: s_n, \ldots, s_{n+k} \in \Z\}$.
\end{lemma}

The further parts of the proof of Theorem~B proceed exactly as in Subsection~\ref{subsec:near_inf} and Sections~\ref{sec:rel}--\ref{sec:est_der} for the attracting case. The only difference is that the curves $\gamma_s$ from Section~\ref{sec:top struct} are degenerated to the points $a_s$. Note that since we assume $\zeta \notin \overline{\Sing(f)}$, the maps $G_s$, $s \in \Z$, extend univalently to $\{z \in \C: \dist(z, \Gamma) < \delta\}$ for some $\delta>0$, which enables to prove an analogue of Lemma~\ref{lem:diameta} similarly as in the attracting case. The details are left to the reader.

\section{Proof of Theorem~C}
\label{sec:bounded}

In \cite{przyttycki86} (see also \cite{PUZ1}), Przytycki introduced a notion of RB-domains (RB standing for `repelling boundary') to generalize simply connected  components of attracting basins of hyperbolic rational maps to the case of maps that are not necessarily globally defined.

\begin{definition}[{\bf RB-domain}{}]\label{defn:RB}
Let $U\subset \mathbb{C}$ be a simply connected domain and suppose that $f$ is holomorphic on a neighbourhood $W$ of $\partial U$. We say that $U$ is an \emph{RB-domain} for $(f, W)$, if 
$f(W\cap U)\subset U$, $f(\partial U)=\partial U$ and $\partial U$ \emph{repels} to the side of $U$, that is, 
\[
\bigcap_{n=0}^\infty f^{-n}(W\cap \overline{U})=\partial U.
\]
\end{definition}

The following theorem by Przytycki \cite{przytycki06} is an extension of the results proved in \cite{zdunik90,zdunik91}.

\begin{theorem}[{\cite[Theorem~A']{przytycki06}}]
\label{thm:przytycki} Let $U\subset \mathbb{C}$ be a simply connected domain and suppose that $f$ is holomorphic on a neighbourhood $W$ of $\partial U$, such that $U$ is an RB-domain for $(f, W)$. Then either $\dimhyp \partial U>1$ or $\partial U$ is an analytic Jordan curve or an analytic arc.
\end{theorem}

The following lemma shows that the notion of RB-domains can be applied in the context of transcendental maps.

\begin{lemma}
\label{lem:RB-domain}
Let $U$ be a bounded invariant component of an attracting basin of a transcendental entire function $f$. Then $U$ is an RB-domain for $(f,W)$,  where $W$ is a small neighbourhood of $\bd U$. 
\end{lemma}
\begin{proof}
Since $U$ is simply connected and bounded, Proposition~\ref{prop:swiss} implies that $f$ on $U$ is conformally conjugated by a Riemann map to a Blaschke product of a finite degree on the unit disc $\D$, with an attracting fixed point $z_0 \in \D$, which easily provides the assertion of the lemma. 
\end{proof}

Observe that if $f$ is a transcendental entire function and $U$ is a component of the attracting basin of $f$, then $\partial U$ cannot be an arc. Thus, in view of Lemma~\ref{lem:RB-domain} and Theorem~\ref{thm:przytycki}, to prove Theorem~C (assuming $p=1$ without loss of generality) it is enough to show the following fact.

\begin{lemma}[{\bf Brolin--Azarina lemma for attracting basins of transcendental entire maps}{}]\label{lem:bounded}
If $U$ is a bounded invariant component of the attracting basin of a transcendental entire function $f$, then $\bd U$ is not an analytic Jordan curve.
\end{lemma}

As explained in Section~\ref{sec:intro}, this fact follows from Azarina's result \cite[Theorem]{Azarina89}. As her work is not widely known within the mathematical community, for the reader's convenience we include a version of the proof of Lemma~\ref{lem:bounded} in the appendix. For a significant part, our arguments follow closely those in the proof of \cite[Lemma~9.1]{brolin65} and \cite[Theorem]{Azarina89}, with some modifications in the final part.

\appendix
\section{Proof of the Brolin--Azarina lemma for transcendental attracting basins}

Assume that $U$ is a bounded invariant component of the attracting basin of a transcendental entire function $f$, and suppose to the contrary that $\partial U$ is an analytic Jordan curve. Let $\zeta\in U$ be the attracting fixed point of $f$. Consider a Riemann map $\varphi$ from $\D$ onto $U$
such that $\varphi(0)=\zeta$. Since $\partial U$ is a Jordan curve, the map $\varphi$ extends homeomorphically to $\overline{\D}$. Furthermore, by the analyticity of $\partial U$ and Schwarz reflection, $\varphi$ extends to a univalent map
\[
\varphi \colon \D_{\rho} \to \C,
\]
where $\D_{\rho} = \D(0,\rho)$ for some $\rho>1$. By Proposition~\ref{prop:swiss}, the map \begin{equation}
\label{eq:phi3}
B = \varphi^{-1} \circ f \circ \varphi
\end{equation}
is a Blaschke product of degree $d$ for some integer $d>1$, with an attracting fixed point at $0$, that is
\[
B(z)=a z^m \prod_{k=1}^{d-m}\frac{z - a_k}{1-\overline{a_k}z},
\]
where $1\leq m\leq d$, $a \in \C$, $\vert a\vert=1$ and $a_k \in \mathbb{D}$. Note that $B$ extends to the Riemann sphere $\clC$ as a rational map of degree $d$, with an attracting fixed point at $\infty$, such that $B(1/\overline{z}) = 1/\overline{B(z)}$, $z \in \C \setminus\{0\}$. In particular, $B(\clC \setminus \overline{\D}) \subset \clC \setminus \overline{\D}$ and 
\begin{equation}\label{eq:toDrho}
\text{for every compact set }K \subset \C \setminus \overline{\D} \text{ there exists } n \ge 0, \text{ such that }B^{-n}(K) \subset \D_\rho \setminus \overline{\D}.
\end{equation}
Moreover, the Schwarz--Pick lemma implies $B^{-1}(\D_\rho \setminus \overline{\D}) \subset \D_\rho \setminus \overline{\D}$.
Note that by \eqref{eq:phi3}, if $\rho$ is chosen sufficiently close to~$1$, then
\begin{equation}
\label{eq:semiconj}
\varphi\circ B = f \circ \varphi \quad \text{on } \D_{\rho}. 
\end{equation}
See Figure~\ref{fig:bounded-semiconj}.

\begin{figure}[ht!]
\includegraphics[width=0.65\textwidth]{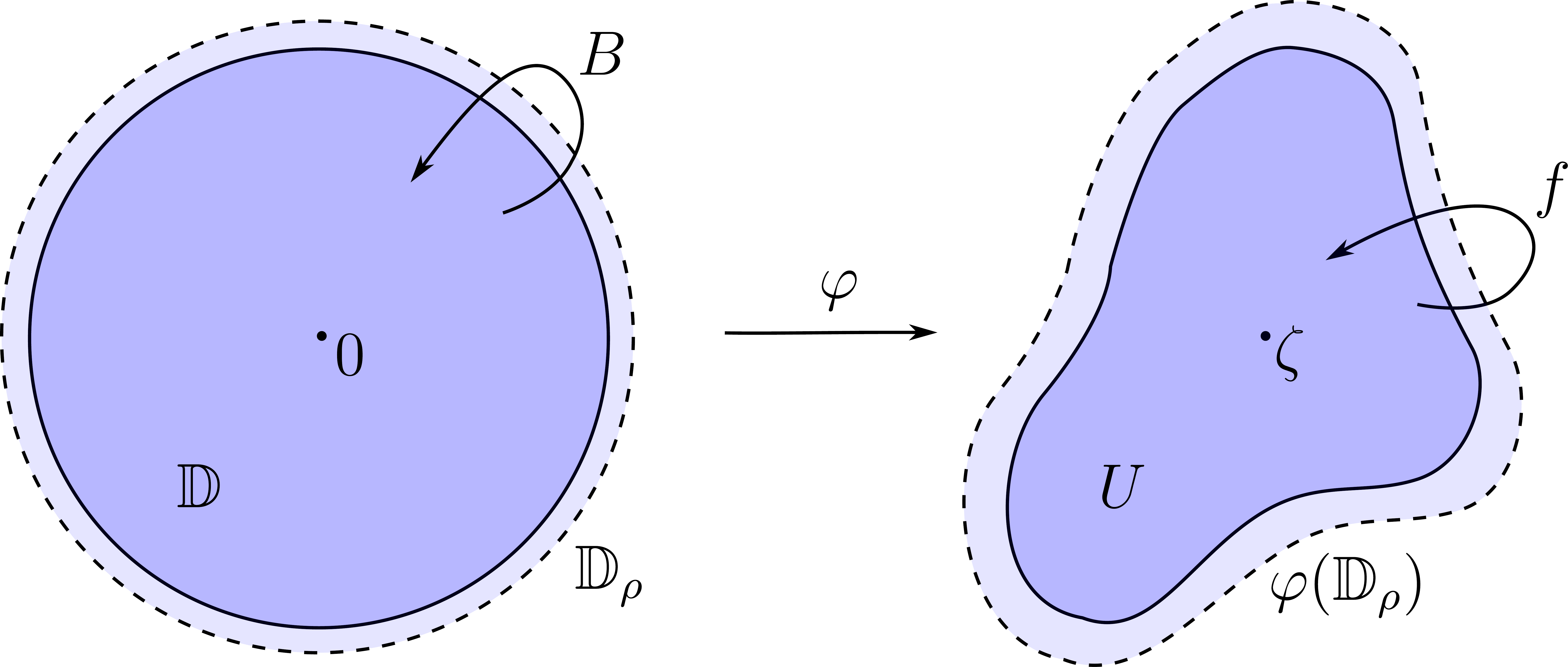}
\caption{The relation between the maps $B$ and $f$.}
\label{fig:bounded-semiconj}
\end{figure}

The rest of the proof is divided into several parts (claims).

\medskip
\noindent
\textit{Claim~$1$}. The map $\varphi$ can be extended to a transcendental entire function $\varphi\colon \C \to \C$ satisfying $\varphi\circ B = f \circ \varphi$. In particular, $B$ has no finite poles.\vspace{-5pt}\\

\noindent
\textit{Proof of Claim~$1$}. First, we extend $\varphi$ to $\C \setminus \mathcal P$, where $\mathcal P$ is the union of the forward orbits of the critical values of $B$ contained in $\C \setminus \overline{\D}$. Note that $\mathcal P$ is at most countable discrete subset of $\C \setminus \D_\rho$ (provided $\rho$ is chosen sufficiently close to~$1$). In particular, $\C \setminus (\overline{\D}\cup \mathcal P)$ is a domain containing $\D_\rho \setminus \overline{\D}$.

Let $z \in \C \setminus (\overline{\D}\cup \mathcal P)$. Note that all inverse branches of $B^n$, $n \in \N$, are defined in a neighbourhood $W$ of $z$. Hence, by \eqref{eq:toDrho}, there exists $n \ge 0$ and an inverse branch $g$ of $B^n$ defined on $W$, such that $g(z) \in \D_\rho \setminus \overline{\D}$. We set 
\[
\varphi = f^{n}\circ \varphi \circ g \quad  \text{on } W.
\]
We show that $\varphi$ is a well-defined holomorphic function on $\C \setminus (\overline{\D}\cup \mathcal P)$. For this, we should check that if $z \in \C \setminus (\overline{\D}\cup \mathcal P)$ and $g_1, g_2$ are inverse branches of $B^{n_1}$ and $B^{n_2}$, respectively, defined in a neighbourhood of $z$, such that $g_1(z), g_2(z) \in \D_\rho \setminus \overline{\D}$, then $\varphi_1(z) = \varphi_2(z)$ for
\[
\varphi_1 = f^{n_1} \circ \varphi \circ g_1, \qquad \varphi_2 = f^{n_2} \circ \varphi \circ g_2.
\]
To see it, connect $z$ to a point $z_0 \in \D_\rho \setminus \overline{\D}$ by a simple curve $\gamma \subset \C \setminus (\overline{\D}\cup \mathcal P)$ and extend the branches $g_1$ and $g_2$ holomorphically to a neighbourhood of $\gamma$. Take $v \in \gamma$ close to $z_0$, such that $v \in \D_\rho \setminus \overline{\D}$, and let $v_1 = g_1(v)$, $v_2 = g_2(v)$. As $B^{-1}(\D_\rho \setminus \overline{\D}) \subset \D_\rho \setminus \overline{\D}$, we have $v_1, B(v_1), \ldots, B^{n_1-1}(v_1) \in \D_\rho$ and $v_2, B(v_2), \ldots, B^{n_2-1}(v_2) \in \D_\rho$. Hence, \eqref{eq:semiconj} implies
\[
\varphi_1(v) = f^{n_1}(\varphi(v_1)) = \varphi(B^{n_1}(v_1)) = \varphi(v) = \varphi(B^{n_2}(v_2)) = f^{n_2}(\varphi(v_2)) = \varphi_2(v).
\]
Thus, by holomorphicity, $\varphi_1 = \varphi_2$ on $\gamma$, in particular $\varphi_1(z) = \varphi_2(z)$. This shows that $\varphi$ is well-defined and holomorphic in $\C \setminus (\overline{\D}\cup \mathcal P)$. It is obvious by definition that $\varphi$ is the extension of the map $\varphi|_{\D_\rho}$ defined previously. In this way we have extended $\varphi$ holomorphically to $\C \setminus \mathcal P$.

To show that $\varphi$ extends holomorphically to $\C$, consider a point $z \in \mathcal P$. Since $\mathcal P$ is discrete, $\varphi$ is defined in $W \setminus \{z\}$ for a small neighbourhood $W$ of $z$, and by \eqref{eq:toDrho}, we have $B^{-n}(W\setminus \{z\}) \subset \D_\rho$ for some $n > 0$, so $\varphi(W\setminus \{z\}) \subset f^n\circ \varphi(\overline{\D_\rho})$ by the definition of $\varphi$. Hence, $\varphi$ is bounded in $W\setminus \{z\}$, so it extends holomorphically to $W$. This shows that $\varphi$ extends to $\C$, so it is an entire function.  

To prove $\varphi\circ B = f \circ \varphi$, consider first $z \in B^{-1}(\C \setminus (\overline{\D}\cup\mathcal P))$. Then $z, B(z) \in \C \setminus (\overline{\D}\cup \mathcal P)$, so by \eqref{eq:toDrho}, there exists an inverse branch $g$ of $B^n$ defined in a neighbourhood of $z$ with $g(z) \in \D_\rho \setminus \overline{\D}$, and an inverse branch $h$ of $B$ defined in a neighbourhood of $B(z)$ with $h(B(z)) = z$. Then by the definition of $\varphi$, 
\[
\varphi(B(z)) = f^{n+1}(\varphi(g(h(B(z))))) =  f(f^n(\varphi(g(z)))) = f(\varphi(z)).
\]
Hence, we have $\varphi\circ B = f \circ \varphi$ on $B^{-1}(\C \setminus (\overline{\D}\cup\mathcal P))$. Together with \eqref{eq:semiconj}, this implies $\varphi\circ B = f \circ \varphi$ on $B^{-1}(\C \setminus \mathcal P)$. Since $B^{-1}(\mathcal P)$ is a discrete subset of $B^{-1}(\C)$, by continuity we have $\varphi\circ B = f \circ \varphi$ on $B^{-1}(\C)$.

Therefore, to show $\varphi\circ B = f \circ \varphi$ on $\C$, it is now sufficient to check $B^{-1}(\C) = \C$, which means that $B$ has no finite poles. To prove it, suppose otherwise and consider a pole $p \in \C$ of $B$. Then $\varphi\circ B = f \circ \varphi$ in $W \setminus \{p\}$ for some neighbourhood $W$ of $p$. Consider a sequence $w_n \in \C$, $w_n \to \infty$, and note that one can find a sequence $z_n \to p$, such that $B(z_n) = w_n$. Since $z_n \in W \setminus \{p\}$ for large $n$, we have 
\[
\varphi(w_n) = \varphi(B(z_n)) = f(\varphi(z_n)) \to f(\varphi(p)).
\]
This shows $\lim_{w\to \infty} \varphi(w) = f(\varphi(p)) \in \C$, so $\varphi$ extends to a rational map on $\clC$, with $\varphi(\infty) \in \C$. This implies $\varphi(\clC) \subset \C$, which is impossible by Liouville's theorem. Consequently, $B$ has no finite poles, which shows $\varphi\circ B = f \circ \varphi$ on $\C$. 

Finally, note that $\varphi$ is transcendental, because otherwise, it is a (non-constant) polynomial $\varphi$ and, consequently, $\varphi\circ B$ is rational, while $f \circ \varphi$ is transcendental, which contradicts the relation $\varphi\circ B = f \circ \varphi$. This ends the proof of the claim.\hspace*{\fill} $\blacktriangle$\\

By Claim~1, $B$ has no finite poles, so $B$ must be of the form $B(z)=az^d$, with $\vert a \vert =1$ and $d > 1$. By applying a linear change of variable we can assume $a=1$, so that
\begin{equation}\label{eq:z^d}
B(z)=z^d.
\end{equation}

\medskip
\noindent
\textit{Claim $2$}. For $r>0$ let
\[
\gamma_r = \varphi(\partial \mathbb{D}_r),
\]
where $\mathbb{D}_r = \mathbb{D}(0,r)$. Then for every $z\in \mathbb{C}\setminus \{0\}$,
\[
\bigcup_{n=0}^{\infty} f^{-n}(f^n(\varphi(z)))
\]
is dense in $\gamma_{|z|}$. Moreover, for every $r>0$, either $\gamma_r\subset J(f)$ or $\gamma_r\subset \mathcal F(f)$. \\

\noindent
\textit{Proof of Claim $2$}. Let $\xi\in\C$ be such that $\xi^{d^n}=1$ for some $n\in\N$. In particular, $\vert \xi z\vert =|z|$ and by \eqref{eq:z^d},
\[
B^n(\xi z)=\xi^{d^n} z^{d^n}=z^{d^n}= B^n(z).
\]
Hence, for each $n\in \N$, there are $d^n$ points of $B^{-n}(B^n(z))$, which are the endpoints of $d^n$ disjoint open arcs of equal length in the circle $\bd\mathbb{D}_{|z|}$, and so the set
\[
\bigcup_{n=0}^{\infty} B^{-n}(B^n(z))
\]
is dense in $\partial \mathbb{D}_{|z|}$. Since $f$ and $B$ are semiconjugate by $\varphi$, the set $\bigcup_{n=0}^{\infty} f^{-n}(f^n(\varphi(z)))$ is dense in $\gamma_{|z|}$, which proves the first part of the claim.

Suppose now $\gamma_r\not\subset \mathcal F(f)$ and take $z\in\gamma_r\cap J(f)$. Then 
\[
\bigcup_{n\geq 0} f^{-n}(f^n(z))\subset J(f)
\]
by the invariance of $J(f)$, and since this union is dense in $\gamma_r$ and $J(f)$ is closed, we obtain $\gamma_r\subset J(f)$. This proves the second part of the claim.
\hspace*{\fill} $\blacktriangle$\\
 
\medskip
\noindent
\textit{Claim $3$}. There is no $z\in\C\setminus\{0\}$ with $\varphi(z)=\zeta$.\\

\noindent
\textit{Proof of Claim $3$}. Suppose to the contrary that there exists $z\in\C\setminus\{0\}$ such that $\varphi(z)=\zeta$. Then $\zeta = \varphi(z) \in \gamma_{|z|}$ and, by Claim~2 and since $\zeta=f^n(\zeta)$ for all $n$, the set 
\[
\bigcup_{n=0}^\infty f^{-n}(\zeta) = \bigcup_{n=0}^\infty f^{-n}(f^n(\varphi(z)))
\]
is dense in $\gamma_{|z|}$. This is a contradiction, because an attracting fixed point cannot have preimages in a small punctured neighbourhood of itself. Hence, 
there is no $z\in\C\setminus\{0\}$ such that $\varphi(z)=\zeta$.
\hspace*{\fill} $\blacktriangle$\\ 

\noindent
\textit{Claim $4$}. The immediate component $U$ of the attracting basin of $\zeta$ is completely invariant.\\

\noindent
\textit{Proof of Claim $4$}.
Suppose $U$ is not completely invariant. Then there exist a component $U_1$ of $f^{-1}(U)$ and a component $U_2$ of $f^{-1}(U_1)$, such that $U, U_1, U_2$ are pairwise disjoint. Since $\varphi$ is entire, there exist $z_1, z_2 \in \C$ such that $\varphi(z_1) \in U_1$, $\varphi(z_2) \in U_2$. Fix $r > 1$ such that $r > \max(|z_1|, |z_2|)$. Then $\overline{U} \subset \varphi(\mathbb{D}_r)$, $U_1$ and $U_2$ intersect $\varphi(\mathbb{D}_r)$ and, by the maximum principle, $\bd \varphi(\mathbb{D}_r) \subset \gamma_r$. Suppose $U_j$ intersects $\gamma_r$ for some $j\in\{1,2\}$. Then by Claim~2, $\gamma_r \subset U_j$, in particular $\bd \varphi(\mathbb{D}_r)\subset U_j$. Hence, $U_j$ is not simply connected, which is impossible, as $U_j$ is a component of a preimage of the simply connected component $U$ under an entire map. Therefore, $U_1 \cup U_2 \subset \varphi(\mathbb{D}_r)$, so $U_1$ and $U_2$ are bounded. By Proposition~\ref{prop:swiss}, we have $f(U_2) = U_1$, $f(U_1) = U$, so one can find points $w_1 \in U_1$, $w_2 \in U_2$, such that $f(w_2) = w_1$, $f(w_1) = \zeta$. 

By Picard's theorem, there exists $z_1 \in \C$ with $\varphi(z_1) = w_1$, or there exists $z_2 \in \C$ with $\varphi(z_2) = w_2$. Note that in the second case the semiconjugation $\varphi\circ B =f \circ \varphi$ implies $\varphi(B(z_2)) = f(w_2) = w_1$, so in fact, the first case always holds. Consequently, by the semiconjugation relation, we have $\varphi(B(z_1)) = \zeta$. By Claim~3, there holds $B(z_1) = 0$, hence \eqref{eq:z^d} implies $z_1 = 0$. Consequently, $U_1 \ni w_1 = \varphi(z_1) = \varphi(0) = \zeta \in U$, which is a contradiction, as $U_1$ and $U$ are disjoint. Hence, $U$ is completely invariant.\hspace*{\fill} $\blacktriangle$\\ 

Finally, observe that by Claim~4, the function $f$ omits all values $w \in U$ on the set $\C \setminus U$, which contains a neighbourhood of infinity, since $U$ is bounded. This provides a contradiction with Picard's theorem, which finishes the proof of the Brolin--Azarina lemma for transcendental attracting basins (Lemma~\ref{lem:bounded}).

\bibliographystyle{plain}
\bibliography{dimhyp}

\end{document}